\documentclass{article}
\usepackage[utf8]{inputenc}
\usepackage{amsmath}
\usepackage{bm}
\usepackage[all,knot]{xy}
\usepackage{amssymb}
\usepackage[alphabetic]{amsrefs}
\usepackage{textcomp}
\usepackage{mathtools}
\usepackage{ytableau}
\ytableausetup
{mathmode, boxsize=2em}
\usepackage[hang]{footmisc}
\usepackage{multirow}
\usepackage{hyperref}
\usepackage{ dsfont }
\usepackage{ stmaryrd }
\usepackage{tipa}
\usepackage{mathrsfs}
\usepackage{booktabs,siunitx}
\usepackage{ragged2e,tabularx, makecell}
\usepackage{scrextend}
\usepackage[a4paper,
            bindingoffset=0.2in,
            left=1.5in,
            right=1.5in,
            top=1in,
            bottom=1in,
            footskip=.25in]{geometry}
\usepackage{enumerate}
\usepackage{multibib}
\usepackage{blkarray}
\usepackage{bm}

\usepackage{lipsum}

\usepackage{blindtext,titlefoot}

\newcites{AppA}{References of Appendix}

\usepackage{amsthm}
\newtheorem{theorem}{Theorem}[section]
\newtheorem*{theorem*}{Theorem}
\newtheorem*{lemma*}{Lemma}
\newtheorem*{fact*}{Fact}
\newtheorem*{conjecture*}{Conjecture}
\newtheorem{proposition}[theorem]{Proposition}
\newtheorem{lemma}[theorem]{Lemma}

\newtheorem{corollary}[theorem]{Corollary}

\newtheorem{definition}[theorem]{Definition}

\newtheorem{lemdef}[theorem]{Lemma-Definition}

\theoremstyle{remark}
\newtheorem{remark}[theorem]{Remark}
\newtheorem{example}[theorem]{Example}

\numberwithin{equation}{section}

\begin{document}

\baselineskip=17pt


\title{On successive minimal bases of division points of Drinfeld modules}

\author{Maozhou Huang\footnote{\texttt{huang.m.aa@m.titech.ac.jp}}}

\date{}

\maketitle

\unmarkedfntext{\text{2020\it{  Mathematics Subject Classification}}. Primary 11G09; Secondary 11S15.}

\unmarkedfntext{\text{\it{Key words and phrases}}: Drinfeld modules, Successive minimal bases, higher ramification subgroups, conductors, Szpiro conjecture.}

\begin{abstract}
We define successive minimal bases (SMBs) for the space of $u^{n}$-division points of a Drinfeld $\Bbb{F}_{q}[t]$-module over a local field, where $u$ is a finite prime of $\Bbb{F}_{q}[t]$ and $n$ is a positive integer.
These SMBs share similar properties to those of SMBs of the lattices associated to Drinfeld modules.
We study the relations between these SMBs and those of the lattices.
Finally, we apply the relations to study the explicit wild ramification subgroup action on an SMB of the space of $u^{n}$-division points and show the function field analogue of Szpiro's conjecture for rank $2$ Drinfeld modules under a certain limited situation.
\end{abstract}

\maketitle

\section{Introduction}
\subsection{Notation}\label{s11}
Let us introduce the notation used throughout this paper.
Put $A \coloneqq \Bbb{F}_{q}[t],$ where $\Bbb{F}_{q}[t]$ is the polynomial ring in $t$ over the field $\Bbb{F}_{q}$ whose order is a power of a rational prime $p.$
Let $F$ be a global function field which is a finite extension of the fraction field of $A.$ 
Let $K$ be the completion of $F$ at a prime $w.$
We also let $w$ denote the valuation associated to $K$ normalized so that $w(K^{\times}) = \Bbb{Z}.$
Fix $K^{\mathrm{sep}}$ (resp. $K^{\mathrm{alg}}$) a separable (resp. algebraic) closure of $K.$
Let $\Bbb{C}_{w}$ denote the completion of $K^{\mathrm{alg}}.$
If $w$ is an infinite prime, we also let $\Bbb{C}_{\infty}$ denote $\Bbb{C}_{w}.$

Let $\phi$ be a rank $r$ Drinfeld $A$-module over $K.$
For an element $a$ in $A,$ let $\phi[a]$ be the $A/a$-module of $a$-division points in $K^{\mathrm{sep}}.$
It is a free module of rank $r.$
Fix a positive integer $n$ and a finite prime $u$ of $A,$ i.e., a monic irreducible polynomial $u \in A.$
The main research objects in this paper are successive minimal bases of $\phi[u^{n}]$ defined below.
For $a\in A$ and $x\in \phi[u^{n}],$ write $a\cdot_{\phi} x \coloneqq \phi_{a}(x)$ for the action of $a$ on $x.$ 

If $w$ is an infinite prime, let $\Lambda$ denote the rank $r$ $A$-lattice in $\Bbb{C}_{\infty}$ and $e_{\phi}$ the exponential function from $\Bbb{C}_{\infty}$ to $\Bbb{C}_{\infty}$ associated to $\phi$ via the uniformization.
Here we have considered $\Lambda$ and the domain of $e_{\phi}$ as $A$-modules via the natural embedding $A\to\Bbb{C}_{\infty}.$

If $w$ is a finite prime, we assume throughout this paper that $\phi$ has stable reduction over $K$ and the reduction of $\phi$ has rank $r' \leq r$ unless otherwise specified.
Let $\psi$ denote the rank $r'$ Drinfeld module over $K$ having good reduction, $\Lambda$ the rank $r-r'$ $A$-lattice in $\Bbb{C}_{w},$ and $e_{\phi}$ the exponential function from $\Bbb{C}_{w}$ to $\Bbb{C}_{w}$ associated to $\phi$ via the Tate uniformization (See \cite[Section~7]{Drin} or Section~\ref{s21}).
Here we consider $\Lambda$ and the domain of $e_{\phi}$ as $A$-modules via $\psi,$ i.e., we have the action of $a$ on $\omega$ to be $a \cdot_{\psi} \omega \coloneqq \psi_{a}(\omega)$ for any $a\in A$ and any $\omega$ in $\Lambda$ or $\Bbb{C}_{w}.$

Let $|-|$ denote one of the following functions.
\begin{enumerate}[\rm{(F}1)]
    \item \label{F1} If $w$ is an infinite prime, we have the absolute value $|-|$ on $K$ which extends the absolute value $|-|=q^{\deg(-)}$ on $\Bbb{F}_{q}((\frac{1}{t})).$
    This absolute value may be extended to $\Bbb{C}_{\infty}.$
    \item \label{F2} 
    Assume that $w$ is a finite prime of $F.$
    Following \cite[Section~1]{Gar}, define a function $|-|$ on $K$ by 
    \[\text{for }x\in K\text{, }|x|=\begin{cases}
    (-w(x))^{1/r'} & w(x)<0,\\
    -w(x)^{1/r'} & w(x)\geq 0,\\
    |0|=-\infty & x=0.
    \end{cases}\]
    We may extend this function to $\Bbb{C}_{w}.$
    This function is not an absolute value.
    However, the ultrametic inequality holds. 
    For $x \in \Bbb{C}_{w},$ we still call $|x|$ the absolute value of $x.$ 
\end{enumerate}

\subsection{On SMBs of $u^{n}$-division points}
The main definition is
\begin{definition}\label{d11}
Let $|-|$ denote the function in (F\ref{F1}) or (F\ref{F2}).
We call a family of elements $\{\lambda_{i}\}_{i=1,\ldots,r}$ an \emph{SMB} (\emph{successive minimal basis}) of $\phi[u^{n}]$ if for each $i,$ the elements $\lambda_{1},\ldots,\lambda_{i}$ in $\phi[u^{n}]$ satisfy
\begin{enumerate}[\rm{(}1)]
    \item $\lambda_{1},\ldots,\lambda_{i}$ are $A/u^{n}$-linearly independent;
    \item $|\lambda_{i}|$ is minimal among the absolute values of elements $\lambda$ in $\phi[u^{n}]$ such that $\lambda_{1},\ldots,\lambda_{i-1},\lambda$ are $A/u^{n}$-linearly independent.
\end{enumerate}
\end{definition}

\noindent Here we have imitiated the definition of SMBs of the lattices $\Lambda$ (See \cite[Section~4]{TSI} or \cite[Section~3]{GekGL}).
Let us remark that 
\begin{remark}\label{r11}
(1) in the definition implies that $\{\lambda_{1},\ldots,\lambda_{r}\}$ is an $A/u^{n}$-basis (or a generating set) of $\phi[u^{n}].$
The condition (2) above can be replaced with ``$w(\lambda_{i})$ is the largest among the valuations of elements $\lambda$ in $\phi[u^{n}]$ such that $\lambda_{1},\ldots,\lambda_{i-1},\lambda$ are $A/u^{n}$-linearly independent''.
In Definition~\ref{d12}, we will extend the definition of SMBs of $\phi[u^{n}]$ to the case where $\phi$ does not necessarily have stable reduction over $K$.
\end{remark}

If $w$ is a finite prime, let $u^{-n}\Lambda$ denote the $A$-module consisting of all roots of $\psi_{u^{n}}(X)-\omega$ for all $\omega\in \Lambda.$ 
For any infinite or finite prime $w,$ by the uniformization or the Tate uniformization of $\phi,$ we have an isomorphism of $A/u^{n}$-modules 
\[\mathcal{E}_{\phi}: u^{-n}\Lambda / \Lambda \to \phi[u^{n}]\]
induced by $e_{\phi}.$
Hence one may expect that there are relations between SMBs of $\phi[u^{n}]$ and those of $\Lambda.$

Let $|-|$ denote the absolute value in (F\ref{F1}) (resp. the function in (F\ref{F2})) if $w$ is an infinite prime (resp. a finite prime).
Put $|u^{n}|_{\infty}=q^{\deg(u^{n})}.$ 

\begin{theorem}\label{ti1}
\begin{enumerate}[\rm{(}1)] 
\item Let $w$ be an infinite prime. 
\begin{itemize}
    \item \text{\rm{(Theorem~\ref{t31})}} Let $\{\omega_{i}\}_{i=1,\ldots,r}$ be an SMB of $\Lambda.$
Then the images $e_{\phi}(\omega_{i}/u^{n})$ for $i=1,\ldots,r$ form an SMB of $\phi[u^{n}].$
    \item \text{\rm{(Corollay~\ref{r32}~(1))}} Let $l$ be a positive integer and $\{\eta_{i}\}_{i=1,\ldots,r}$ an SMB of $\phi[u^{l}].$ 
Let $\{\lambda_{i}\}_{i=1,\ldots,r}$ be an SMB of $\phi[u^{n}].$ Assume that $n$ satisfies $|u^{n}|_{\infty}>|\eta_{r}|/|\eta_{1}|.$ 
Under this assumption, for each $i=1,\ldots,r,$ the element $\lambda_{i}$ has only one preimage under $e_{\phi},$ denoted $\log_{\phi}(\lambda_{i}),$ with absolute value $<|\omega|$ for any $\omega \in \Lambda \setminus \{0\}.$ 
Then the family of elements $\{u^{n}\log_{\phi}(\lambda_{i})\}_{i=1,\ldots,r}\subset \Bbb{C}_{\infty}$ is an SMB of $\Lambda.$ 
\end{itemize}
\item Let $w$ be a finite prime.  
\begin{itemize}
    \item \text{\rm{(Theorem~\ref{t41})}} Let $\{\omega_{i}\}_{i=1,\ldots,r'}$ \text{\rm{(}}resp. $\{\omega_{i}^{0}\}_{i=r'+1,\ldots,r}$\text{\rm{)}} be an SMB of $\psi[u^{n}]$ \text{\rm{(}}resp. $\Lambda$\text{\rm{)}}.
    Let $\omega_{i}$ be a root of $\psi_{u^{n}}(X)-\omega_{i}^{0}$ for $i=r'+1,\ldots,r.$
    Then the images $e_{\phi}(\omega_{i})$ for $i=1,\ldots,r$ form an SMB of $\phi[u^{n}].$
    \item \text{\rm{(Corollary~\ref{r41}~(1) and (2))}} Let $l$ be a positive integer and $\{\eta_{i}\}_{i=1,\ldots,r}$ an SMB of $\phi[u^{l}].$
Let $\{\lambda_{i}\}_{i=1,\ldots,r}$ be an SMB of $\phi[u^{n}].$
Assume that $n$ satisfies $|u^{n}|_{\infty}>|\eta_{r}|/|\eta_{r'+1}|.$
Under this assumption, for each $i=1,\ldots,r,$ the element $\lambda_{i}$ has only one preimage, denoted $\log_{\phi}(\lambda_{i}),$ with absolute value $<|\omega|$ for any $\omega \in \Lambda \setminus \{0\}.$ 
Then the family of elements $\{\log_{\phi}(\lambda_{i})\}_{i=1,\ldots,r'}\subset \Bbb{C}_{w}$ \text{\rm{(}}resp.  $\{u^{n}\cdot_{\psi}\log_{\phi}(\lambda_{i})\}_{i=r'+1,\ldots,r}\subset \Bbb{C}_{w}$\text{\rm{)}} is an SMB of $\psi[u^{n}]$ \text{\rm{(}}resp. of $\Lambda$\text{\rm{)}}.
\end{itemize}
\end{enumerate}
\end{theorem}

It turns out that the SMBs of $\phi[u^{n}]$ has the following properties.
\begin{proposition}\label{p11}
Let $\{\lambda_{i}\}_{i=1,\ldots,r}$ be an SMB of $\phi[u^{n}].$
\begin{enumerate}[\rm{(}1)]   
\item \text{\rm{(Proposition~\ref{p221})}} 
The sequence $|\lambda_{1}|\leq |\lambda_{2}|\leq \cdots \leq |\lambda_{r}|$ associated to an SMB of $\phi[u^{n}]$ is an invariant of $\phi[u^{n}],$ i.e., for any SMB $\{\lambda_{i}'\}_{i=1,\ldots,r}$ of $\phi[u^{n}],$ we have $|\lambda_{i}'|=|\lambda_{i}|$ for all $i.$
\item \text{\rm{(Proposition~\ref{p32} and \ref{p42})}} 
Assume that $u$ is not divisible by the prime $w,$ i.e., $w(u)\leq 0.$ 
Then we have
\[\bigg|\sum_{i}a_{i}\cdot_{\phi} \lambda_{i}\bigg|=\max_{i}\{|a_{i}\cdot_{\phi} \lambda_{i}|\}\]
for any $a_{i}\in A\mod u^{n}.$
\item \text{\rm{(Proposition~\ref{p222})}} 
There exists an SMB $\{\lambda_{i}'\}_{i=1,\ldots,r}$ of $\phi[u^{n+1}]$ such that $u \cdot_{\phi} \lambda_{i}'=\lambda_{i}$ for all $i.$ The elements $u\cdot_{\phi}\lambda_{i}$ for $i=1,\ldots,r$ form an SMB of $\phi[u^{n-1}].$
\end{enumerate}
\end{proposition}

\noindent 
Here the properties (1) and (2) are similar to those of SMBs of lattices (See Propositions~\ref{p211} and \ref{p212}).
We remark that (2) essentially follows from similar properties of SMBs of lattices (See Proposition~\ref{p212} or \cite[Lemma~4.2]{TSI}).
We hope to know whether the condition ``$w(u) \leq 0$'' in (2) can be removed.

Let $K(\Lambda)$ (resp. $K(u^{-n}\Lambda)$ and $K(\phi[u^{n}])$) denote the extension of $K$ generated by all elements in $\Lambda$ (resp. $u^{-n}\Lambda$ and $\phi[u^{n}]$).
By Theorem~\ref{ti1}, we are able to show 
\begin{proposition}\label{p12}
Let $l$ be a positive integer and $\{\eta_{i}\}_{i=1,\ldots,r}$ an SMB of $\phi[u^{l}].$ 
Let $\{\lambda_{i}\}_{i=1,\ldots,r}$ be an SMB of $\phi[u^{n}].$
\begin{enumerate}[\rm{(}1)]
    \item \text{\rm{(Corollary~\ref{r32}~(2))}} If $w$ is an infinite prime and $n$ is large enough so that $|u^{n}|_{\infty}>|\eta_{r}|/|\eta_{1}|,$ then we have $K(\Lambda)=K(\phi[u^{n}]).$
    \item \text{\rm{(Corollary~\ref{r41}~(3))}} If $w$ is a finite prime and $n$ is large enough so that $|u^{n}|_{\infty}>|\eta_{r}|/|\eta_{r'+1}|,$ then we have $K(u^{-n}\Lambda)=K(\phi[u^{n}]).$
\end{enumerate}
\end{proposition}
\noindent The claim (1) is an effective version of \cite[Proposition~2.1]{Mau}.

\subsection{An application to rank 2 Drinfeld modules}\label{s13}
Let $u$ be a finite prime of $A.$ 
Let $\phi$ be a rank $2$ Drinfeld $A$-module over $K$ which does not necessarily have stable reduction when $w$ is finite.
Let $\{\lambda_{i}\}_{i=1,2}$ be an SMB of $\phi[u^{n}]$.
Let $\bm{j}$ denote the $j$-invariant of $\phi.$
Assume \begin{equation}\begin{cases}\begin{split}& \text{either }\big(w(\bm{j}) < w(t)q\text{ and }p \nmid w(\bm{j})\big),\\
& \text{or }w(\bm{j})\geq w(t)q\end{split} & \text{ if }w\text{ is infinite};\\
\begin{split}& \text{either }\big(w(\bm{j})<0\text{ and }p\nmid w(\bm{j})\big),\\
& \text{or }w(\bm{j}) \geq 0\end{split} & \text{ if }w\text{ is finite}.
\end{cases}\label{f11}\end{equation}
For a positive integer $n,$ let $G(n)_{1}$ denote the wild ramification subgroup, i.e., the first lower ramification subgroup, of $\mathrm{Gal}(K(\phi[u^{n}])/K).$
In \cite[Theorems~3.9 and 3.13, Lemmas~3.14 and 3.15]{AH22}, for $u$ having degree $1$ and any $n,$ the action of $G(n)_{1}$ on $\{\lambda_{i}\}_{i=1,2}$ has been studied assuming moreover $q \neq 2$ when $w$ is a finite prime.
In Sections~\ref{s5} and \ref{s6}, we study the action of $G(n)_{1}$ on $\{\lambda_{i}\}_{i=1,2}$ \emph{without} requiring ($\deg(u)=1$) and ($q \neq 2$ when $w$ is finite).

\begin{theorem}\label{ti2}
Let $\phi$ be a rank $2$ Drinfeld $A$-module over $K$ which does not necessarily have stable reduction when $w$ is finite.
Let $u$ be a finite prime of $A$ with $\deg(u) = d.$
Let $\{\lambda_{i}\}_{i=1,2}$ be an SMB of $\phi[u^{n}].$ 
\begin{enumerate}[\rm{(}1)]
\item \text{\rm{(Theorem~\ref{t51})}} Let $w$ be an infinite prime. 
Assume $w(\bm{j})<w(t)q$ and $p \nmid w(\bm{j}).$
Let $m$ be the integer such that $w(\bm{j})\in (w(t)q^{m+1}, w(t)q^{m}).$ 
Put $d = \deg(u).$
Assume $n\geq m/d.$
\begin{itemize}
\item Any element in $G(n)_{1}$ fixes $\lambda_{1};$
\item 
Let $A^{<m}$ denote the subgroup of $A$ consisting of elements with degree $<m.$
Then the map \[G(n)_{1} \to A^{<m}\cdot_{\phi}\lambda_{1};\,\,\sigma\mapsto \sigma(\lambda_{2}) - \lambda_{2}\]
is an isomorphism of groups.
\end{itemize}
\item \text{\rm{(Corollary~\ref{c61})}} Let $w$ be a finite prime satisfying $w \nmid u.$ 
Assume $w(\bm{j}) < 0$ and $p\nmid w(\bm{j}).$ 
\begin{itemize}
\item Any element in $G(n)_{1}$ fixes $\lambda_{1};$
\item 
There is an isomorphism of groups
\[G(n)_{1} \to A\cdot_{\phi}\lambda_{1};\,\,\sigma \mapsto \sigma(\lambda_{2}) - \lambda_{2}.\]
\end{itemize}
\end{enumerate}
\end{theorem}

Example~\ref{e51} provides an instance where $w$ is an infinite prime, $w(\bm{j}) < w(t)q,$ $p \mid w(\bm{j}),$ and the extension $K(\phi[u^{n}])/K$ is not wildly ramified.
Let us remark that (1) if $w$ is an infinite prime and $w(\bm{j}) \geq w(t)q,$ the extension $K(\phi[u^{n}])/K$ is at worst tamely ramified such that $G(n)_{1}$ is a trivial group for any $n \geq 1;$
(2) if $w$ is a finite prime and $w(\bm{j})\geq 0,$ then $\phi$ has potentially good reduction at $w$ such that the extension $K(\phi[u^{n}])/K$ is at worst tamely ramified and the group $G(n)_{1}$ is trivial for any $n \geq 1.$

Let $\phi$ be a rank $2$ Drinfeld $A$-module over $F.$
With the assumptions on its $j$-invariant in (\ref{f11}), we define and calculate the conductors of $\phi$ at each prime $w$ of $F$ using the $u$-adic Tate module with $u \nmid w.$
Finally, we show a function field analogue of Szpiro's conjecture in Theorem~\ref{t62}, which slightly generalizes \cite[Theorem~4.3]{AH22}.

Motivated by \cite[Proposition~3.2]{GekP}, we may expect that there are generalizations of the results in Sections~\ref{s5} and \ref{s6} to Drinfeld $A$-modules $\phi$ of rank $r$ over $K$ satisfying $\phi_{t}(X) = tX + a_{s}X^{q^{s}} + a_{r}X^{q^{r}} \in K[X].$
We have obtained a generalization of Proposition~\ref{p51} for such $\phi$ (See Remark~\ref{r52}).
There are difficulties in generalizing Theorem~\ref{ti2}.
We do not further investigate the general case in the present paper. 
Some partial results will appear in the author's doctoral thesis \cite{H23}.
For instance, the explicit action of the ramification subgroup $\mathrm{Gal}(K(\phi[t])/K)_{1}$ on an SMB $\{\xi_{i}\}_{i=1,\ldots,r}$ of $\phi[t]$ has been worked out in \cite[Theorem~3.3.16]{H23} under certain limited situations.

\subsection{Contents}
Except for Section~\ref{s63}, we consider Drinfeld $A$-modules over a localization $K$ of a global function field.
In Section~\ref{s2}, we first review the basics of the SMB of lattices. 
The rest of this section is devoted to the basics of SMBs of $\phi[u^{n}].$
In Section~\ref{s3}, we mainly show the infinite prime case of Theorem~\ref{ti1}.
For an element $\omega_{i}$ of an SMB of the lattice $\Lambda$ as in Theorem~\ref{ti1}~(1) and an element $a_{i}$ in $A$ with a limited degree, we describe the absolute value of $e_{\phi}(a_{i}\omega_{i})$ in Corollary~\ref{c31}~(1). 
This is the key result of this section and its proof is inspired by that of \cite[Lemma~3.4]{GekGL}.
Section~\ref{s4} consists of finite prime analogues of the results in Section~\ref{s3}.
The analogue of Corollary~\ref{c31}~(1) is Corollary~\ref{c41}~(1).

In Section~\ref{s5} (resp. Section~\ref{s6}), we apply the results in the previous sections to a rank $2$ Drinfeld module $\phi$ over $K$ with $w$ being infinite (resp. finite).
We first calculate the valuations of elements of SMBs of $\Lambda$ and $\phi[u^{n}]$ in Sections~\ref{s51} and \ref{s61}.
In Section~\ref{s52}, we calculate the conductors of $\phi$ in Lemma~\ref{d51}.
Then we study the action of the wild ramification subgroup of the Galois group $\mathrm{Gal}(K(\phi[u^{n}])/K)$ on an SMB of $\phi[u^{n}]$ in Theorem~\ref{t51}.
Section~\ref{s62} consists of the finite prime analogues of the results in Section~\ref{s52}. 
In Section~\ref{s63}, we obtain the function field analogue of Szpiro's conjecture under certain assumptions.

In Appendix~\ref{sa}, when $w$ is an infinite prime, the conductor of a rank $r$ Drinfeld $A$-module over $K$ is defined.
In Appendix~\ref{sb}, we calculate the Herbrand $\psi$-function of the extension of $K$ generated by the roots of a certain polynomial with degree being a power of $q.$

\subsection*{Acknowledgements}
We are very grateful to Yuichiro Taguchi for his constant interest and encouragement. 
His comments have been very efficacious.
We thank M. Papikian for informing us of the paper \cite{GekP} and the initial idea so that we can formulate Remark~\ref{r52}.
We thank the anonymous referee for this constructive comments. 

\section{Basics of SMBs}\label{s2}
Let $|-|$ denote the absolute value in (F\ref{F1}) (resp. the function in (F\ref{F2})) if $w$ is an infinite prime (resp. a finite prime) defined in Section~\ref{s11}.

\subsection{SMBs of lattices}\label{s21}
In this subsection, we recall first the basics of SMBs of lattices and then the (Tate) uniformization of Drinfeld modules. 
Consider $\Bbb{C}_{\infty}$ as an $A$-module via the embedding $A\to \Bbb{C}_{\infty}.$ 
If $w$ is a finite prime, consider $\Bbb{C}_{w}$ as an $A$-module via a Drinfeld module $\psi$ having good reduction of rank $r'.$
The next lemma will be applied implicitly in this paper.

\begin{lemma}\label{l20}
\begin{enumerate}[\rm{(}1)]
\item If $w$ is an infinite prime, we have $|a\omega|=|a|\cdot |\omega|$ for any $a\in A$ and $\omega\in \Bbb{C}_{\infty}.$
\item \text{\rm{(\text{\cite[Section~1]{Gar}})}} Let $w$ be a finite prime. 
Then we have $|a\cdot_{\psi} \omega|=|a|_{\infty}\cdot |\omega|,$ i.e., $w(a\cdot_{\psi} \omega)=|a|_{\infty}^{r'}\cdot w(\omega)$  for any $a\in A$ and any $\omega\in \Bbb{C}_{w}$ having valuation $<0,$ where $|a|_{\infty}=q^{\deg(a)}.$ 
\end{enumerate}
\end{lemma}

\begin{proof}
(1) is clear.
We show (2).
Put $g=r'\cdot \deg(a),$ $a_{0}=a,$ and $\sum_{i=0}^{g}a_{i}X^{q^{i}}=\psi_{a}(X).$
As the Drinfeld module $\psi$ has good reduction, we have $w(a_{i})\geq 0$ and $w(a_{g})=0.$
Hence the assumption $w(\omega)<0$ implies that the valuation $w(a_{g}\omega^{q^{g}})$ is the strictly smallest among $w(a_{i}\omega^{q^{i}})$ for all $i.$
As $w(a_{g})=0,$ we have $w(a_{g}\omega^{q^{g}})=q^{g}w(\omega),$ i.e., $|a \cdot_{\psi}\omega|=|a|_{\infty}\cdot |\omega|.$ 
\end{proof}

Let $L$ be an $A$-lattice of rank $r$ in $\Bbb{C}_{\infty}$ or an $A$-lattice of rank $r$ in $\Bbb{C}_{w}$ such that each nonzero element in the lattice has valuation $<0.$
\begin{definition}[\text{\cite[Section~4]{TSI} or \cite[Section~3]{GekGL}}]\label{d21}
A family of elements $\{\omega_{i}\}_{i=1,\ldots,r}$ in $L$ is called an SMB of $L$ if for each $i,$ the elements $\omega_{1},\ldots,\omega_{i}$ satisfy
\begin{enumerate}
\item $\omega_{1},\ldots,\omega_{i}$ are $A$-linearly independent;
\item 
$|\omega_{i}|$ is minimal among the absolute values of elements $\omega$ in $L$ such that $\omega_{1},\ldots,\omega_{i-1},\omega$ are $A$-linearly independent.
\end{enumerate}
\end{definition}

\begin{remark}
If elements $\lambda_{i}$ for $i=1,\ldots,r$ of $\phi[u^{n}]$ are $A/u^{n}$-linearly independent (cf. Definition~\ref{d11}~(1)), then $\{\lambda_{i}\}_{i=1,\ldots,r}$ is an $A/u^{n}$-basis of $\phi[u^{n}].$
On the other hand, if elements $\omega_{i}$ for $i=1,\ldots,r$ of $\Lambda$ are $A$-linearly independent, then $\{\omega_{i}\}_{i=1,\ldots,r}$ is not necessarily an $A$-basis of $\Lambda.$
\end{remark}

\begin{proposition}\label{p211}
Let $\{\omega_{i}\}_{i=1,\ldots,r}$ be a family of elements in $L.$ 
\begin{enumerate}[\rm{(}1)]
\item This family is an SMB if and only if for each $i,$ the elements $\omega_{1},\ldots,\omega_{i}$ satisfy
\begin{itemize}
\item $\omega_{1},\ldots,\omega_{i}$ are $A$-linearly independent;
\item we have $|\omega_{i}|=l_{i},$ where
\[l_{i}=\min\left\{\rho\in \Bbb{R}\,\middle|\,\begin{aligned}&\text{the ball in }\Bbb{C}_{\infty}\text{ or }\Bbb{C}_{w}\text{ around }0\text{ of radius }\rho\text{ contains}\\
&\text{at least }i\text{ elements in }L\text{ which are }A\text{-linearly independent}&
\end{aligned}\right\}.\]
\end{itemize}
\item The sequence $|\omega_{1}| \leq |\omega_{2}| \leq \cdots \leq |\omega_{r}|$ for an SMB $\{\omega_{i}\}_{i=1,\ldots,r}$ is an invariant of $L,$ i.e., for any SMB $\{\omega_{i}'\}_{i=1,\ldots,r}$ of $L,$ we have $|\omega_{i}|=|\omega_{i}'|$ for all $i.$
\end{enumerate}
\end{proposition}

\begin{proposition}\label{p212}
Let $\{\omega_{i}\}_{i=1,\ldots,r}$ be a family of elements in $L$ so that $|\omega_{1}|\leq |\omega_{2}| \leq \cdots\leq |\omega_{r}|$.  
Then this family is an SMB of $L$ if and only if 
\begin{enumerate}[\rm{(}1)]
\item $\omega_{1},\ldots,\omega_{r}$ form an $A$-basis of $L;$
\item we have $|\sum_{i}a_{i}\omega_{i}|=\max_{i}\{|a_{i}\omega_{i}|\}$ for any $a_{i}\in A.$
\end{enumerate}
\end{proposition}

\begin{proof}
This has been proved in \cite[Lemma~4.2]{TSI}.
\end{proof}

For the subfield $K$ of $\Bbb{C}_{v},$ we say that $L$ is $\mathrm{Gal}(K^{\mathrm{sep}}/K)$-invariant if each element in the Galois group maps $L$ into $L.$
The following lemma concerns the extension generated by elements $\omega$ in the lattice with $|\omega|$ being minimal.

\begin{lemma}\label{l21}
Let $\{\omega_{i}\}_{i=1,\ldots,r}$ be an SMB of $L$ such that $|\omega_{1}| = \cdots = |\omega_{s}| < |\omega_{s+1}|$ for some positive integer $s<r.$
Assume that
\begin{itemize}
\item
the extension $M/K$ generated by $\omega_{i}$ for $i=1,\ldots,s$ is separable;
\item
the lattice $L$ is $\mathrm{Gal}(K^{\mathrm{sep}}/K)$-invariant.
\end{itemize}
\begin{enumerate}[\rm{(}1)]
    \item 
The extension $M/K$ is Galois.
    \item 
The extension $M/K$ is at worst tamely ramified.
\end{enumerate}
\end{lemma}

\begin{proof}
We show (1).
Let $\widehat{M}$ denote the Galois closure of $M/K$ so that $\widehat{M}$ is exactly the compositum of $\varsigma M$ for all $\varsigma\in \mathrm{Gal}(\widehat{M}/K).$
We have $\widehat{M} = M.$
Indeed, if $\widehat{M}/M$ is nontrivial, there exists some element $\varsigma \in \mathrm{Gal}(\widehat{M}/K)$ such that $\varsigma(\omega_{j}) \notin M$ for $j$ to be one of $1,\ldots,s.$
Note that $M$ contains the $A$-module $\bigoplus_{i=1,\ldots,s} A\omega_{i}$ (here $A\omega_{i} \coloneqq \{a \cdot_{\psi} \omega_{i} \mid a\in A\}$ if the prime $w$ is finite). 
As elements in $\Lambda\setminus\bigoplus_{i=1,\ldots,s} A\omega_{i}$ have strictly smaller valuations than that of $\omega_{i}$ for $i=1,\ldots,s$ and Galois actions preserve valuations, this implies that $\varsigma(\omega_{j}) \notin L.$ 
If $\varsigma$ also denotes a preimage of $\varsigma$ under $\mathrm{Gal}(K^{\mathrm{sep}}/K) \to \mathrm{Gal}(\widehat{M}/K),$ then  $\varsigma(\omega_{j}) \notin L$ contradicts that $L$ is $\mathrm{Gal}(K^{\mathrm{sep}}/K)$-invariant.

As for (2), we show that $M/K$ is tamely ramified.
Assume the converse so that the wild ramification subgroup $\mathrm{Gal}(M/K)_{1}$ is nontrivial.
Let $w_{M}$ denote the normalized valuation associated to $M.$
For $\sigma$ to be a nontrivial element in $\mathrm{Gal}(M/K)_{1},$ we have for each $i$ 
\begin{align}1 \leq w_{M}(\sigma(\omega_{i})\omega_{i}^{-1}-1). \nonumber\end{align}
We also have $\sigma(\omega_{j}) - \omega_{j} \neq 0$ for $j$ to be one of $1,\ldots,s.$ 
Note that $w_{M}(\omega_{j})$ is the largest among the valuations of all nonzero elements in $L.$
As $\sigma(\omega_{j})-\omega_{j} \in L$ ($L$ is $\mathrm{Gal}(K^{\mathrm{sep}}/K)$-invariant), we have 
\[w_{M}(\sigma(\omega_{j})\omega_{j}^{-1}-1) = w_{M}(\sigma(\omega_{j})-\omega_{j}) - w_{M}(\omega_{j}) \leq 0.\]
This gives a contradiction.
\end{proof}

Next, we recall the uniformization and the Tate uniformization.
If $w$ is an infinite prime, then the uniformization associates to the Drinfeld module $\phi$ a $\mathrm{Gal}(K^{\mathrm{sep}}/K)$-invariant $A$-lattice $\Lambda$ and an exponential function $e_{\phi}$ on $\Bbb{C}_{\infty}$ such that for each $a\in A,$  the following diagram commutes, and its two rows are short exact sequences
\begin{equation}
\begin{split}
\xymatrix{
\Lambda\ar@{^{(}->}[r]\ar[d]^{a}& \Bbb{C}_{\infty}\ar[r]^{e_{\phi}}\ar[d]^{a}&\Bbb{C}_{\infty}\ar[d]^{\phi_{a}}\\
\Lambda\ar@{^{(}->}[r]& \Bbb{C}_{\infty}\ar[r]^{e_{\phi}}&\Bbb{C}_{\infty}.
}    
\end{split}\nonumber
\end{equation}
Here the exponential function is explicitly
\[e_{\phi}: \Bbb{C}_{\infty}\to \Bbb{C}_{\infty};\,\,\omega\mapsto \omega\prod_{\mu\in \Lambda\setminus\{0\}}(1-\omega/\mu)\]
and the coefficients of $\phi_{a}(X)$ map to $\Bbb{C}_{\infty}$ via the embedding $K\hookrightarrow \Bbb{C}_{\infty}.$
The commutativity of the right square in the diagram means $e_{\phi}(a\omega)=a \cdot_{\phi} e_{\phi}(\omega)$ for any $\omega\in \Bbb{C}_{\infty}.$

\begin{remark}[SMBs and isomorphic Drinfeld modules]\label{r21}
For any $b\in K^{\mathrm{sep}}\setminus\{0\},$  we have the Drinfeld module $b\phi b^{-1}$ isomorphic to $\phi.$
The uniformization associates to $b\phi b^{-1}$ the lattice $b\Lambda.$
If the family $\{\omega_{i}\}_{i=1,\ldots,r}$ is an SMB of $\Lambda,$ then $\{b\omega_{i}\}_{i=1,\ldots,r}$ is an SMB of $b\Lambda.$ 
\end{remark}

If $w$ is a finite prime of $K,$ assume that $\phi$ has stable reduction over $K$ and the reduction of $\phi$ has rank $r'<r.$
According to \cite[Section~7]{Drin}, there are the following data associated to $\phi:$
\begin{enumerate}[\rm{(}1)]
\item A rank $r'$ Drinfeld $A$-module $\psi$ over $K$ has good reduction;
\item  A $\mathrm{Gal}(K^{\mathrm{sep}}/K)$-invariant $A$-lattice $\Lambda$ has rank $r-r'$ with the $A$ action induced by $\psi.$ 
Each element of $\Lambda$ has valuation $<0.$
\item An analytic entire surjective homomorphism \[e_{\phi}:\Bbb{C}_{w}\to \Bbb{C}_{w};\,\,\omega\mapsto \omega\prod_{\mu\in \Lambda\setminus\{0\}}(1-\omega/\mu)\] such that for each $a\in A,$ the following diagram commutes, and its two rows are short exact sequences
\[\xymatrix{\Lambda\ar@{^{(}->}[r]\ar[d]^{\psi_{a}}&\Bbb{C}_{w}\ar[r]^{e_{\phi}}\ar[d]^{\psi_{a}} & \Bbb{C}_{w}\ar[d]^{\phi_{a}}\\
\Lambda\ar@{^{(}->}[r] & \Bbb{C}_{w}\ar[r]^{e_{\phi}}& \Bbb{C}_{w}.
}\]
The commutativity of the right square means $e_{\phi}(a \cdot_{\psi} \omega)=a \cdot_{\phi} e_{\phi}(\omega)$ for any $\omega\in \Bbb{C}_{w}.$
\end{enumerate}

We call these data the Tate uniformization of $\phi.$ 

\subsection{SMBs of the module of $u^{n}$-division points}\label{s22}
In this subsection, let $\phi$ be a  rank $r$ Drinfeld $A$-module over $K$ which does not necessarily have stable reduction.
Using Remark~\ref{r11}, we may extend Definition~\ref{d11}.
\begin{definition}[Extending Definition~\ref{d11}]\label{d12}
Let $n$ be a positive integer and $u$ a finite prime of $A.$
A family of elements $\{\lambda_{i}\}_{i = 1,\ldots,r}$ is an \emph{SMB} of $\phi[u^{n}]$ if for each $i,$ the elements $\lambda_{1},\ldots,\lambda_{i}$ in $\phi[u^{n}]$ satisfy 
\begin{enumerate}[\rm{(}1)]
    \item 
$\lambda_{1},\ldots,\lambda_{i}$ are $A/u^{n}$-linearly independent;
    \item 
$w(\lambda_{i})$ is the largest among the valuations of elements $\lambda$ in $\phi[u^{n}]$ such that $\lambda_{1},\ldots,\lambda_{i-1},\lambda$ are $A/u^{n}$-linearly independent.
\end{enumerate}
\end{definition}

\begin{remark}\label{r22}
For any $b\in K^{\mathrm{sep}}\setminus\{0\},$ a family $\{\lambda_{i}\}_{i=1,\ldots,r}$ is an SMB of $\phi[u^{n}]$ if and only if the family $\{b\lambda_{i}\}_{i=1,\ldots,r}$ is an SMB of $b\phi b^{-1}[u^{n}].$ 
Especially, this holds when $w$ is a finite prime and $b$ is an element in some tamely ramified extension $L$ of $K$ so that $b\phi b^{-1}$ has stable reduction over $L.$
\end{remark}

The rest of this subsection is concerned with two basic properties of SMBs of $\phi[u^{n}].$ 

\begin{proposition}[cf. Proposition~\ref{p211}]\label{p221}
Let $\{\lambda_{i}\}_{i=1,\ldots,r}$ be a family of elements in $\phi[u^{n}].$
\begin{enumerate}[\rm{(}1)]
\item  This family is an SMB if and only if for each $i,$ the elements $\lambda_{1},\ldots,\lambda_{i}$ satisfy
\begin{itemize}
\item $\lambda_{1},\ldots,\lambda_{i}$ are $A/u^{n}$-linearly independent;
\item we have $w(\lambda_{i})=l_{i},$ where 
\[l_{i} = \max \left\{ \rho\in \Bbb{R}\,\middle|\,\begin{aligned}&\text{the ball }\{\lambda\in K^{\mathrm{sep}}\mid w(\lambda) \geq \rho\}\text{ contains at least }i\\
&\text{elements in }\phi[u^{n}]\text{ which are }A/u^{n}\text{-linearly independent}
\end{aligned}
\right\}.\]
\end{itemize}
\item The sequence $w(\lambda_{1}) \geq w(\lambda_{2}) \geq \cdots \geq w(\lambda_{r})$ for an SMB $\{\lambda_{i}\}_{i=1,\ldots,r}$ is an invariant of $\phi[u^{n}].$ 
\end{enumerate}
\end{proposition}

Assume that $\phi$ has stable reduction when $w$ is finite.
Then the sequence $|\lambda_{1}| \leq |\lambda_{2}| \leq \cdots \leq |\lambda_{r}|$ for an SMB $\{\lambda_{i}\}_{i=1,\ldots,r}$ is an invariant of $\phi[u^{n}].$

\begin{proof}
(2) straightforwardly follows from (1).
We then show (1).
The ``$\Leftarrow$'' is straightforward.
For ``$\Rightarrow$'', the first dot in (1) is the same as Definition~\ref{d12}~(1).
Clearly, we have $l_{i}\geq w(\lambda_{i})$ for all $i$ and $l_{1}=w(\lambda_{1}).$
Then we proceed by induction.
We fix any $i,$ assume $l_{j}=w(\lambda_{j})$ for $j<i,$ and show $l_{i}=w(\lambda_{i}).$
We assume $l_{i}>w(\lambda_{i})$ and find a contradiction.
There exist elements $\eta_{1},\ldots,\eta_{i}\in \phi[u^{n}]$ such that $\eta_{1},\ldots,\eta_{i}$ are $A/u^{n}$-linearly independent and $w(\eta_{j})\geq l_{i} > w(\lambda_{i})$ for $j=1,\ldots,i.$ 

Put $\overline{\eta}_{j} \coloneqq u^{n-1}\cdot_{\phi}\eta_{j}$ for $j \leq i$ and $\overline{\lambda}_{j} \coloneqq u^{n-1}\cdot_{\phi}\lambda_{j}$ for $j < i.$ 
We claim that there is some $k$ such that $\overline{\eta}_{k}$ and $\overline{\lambda}_{1},\ldots,\overline{\lambda}_{i-1}$ are $A/u$-linearly independent. 
Assume the inverse.
Then we have equations
\[b_{l}\cdot_{\phi}\overline{\eta}_{l} + \sum_{j=1}^{i-1}a_{l,j}\cdot_{\phi}\overline{\lambda}_{j} = 0\]
for all $l=1,\ldots,i,$ where $a_{l,j}\in A \mod u$ and $b_{l} \in A \mod u$ with $b_{l}\not\equiv 0 \mod u$ for each $l.$   
For each $l,$ we obtain
\[\overline{\eta}_{l} = \sum_{j=1}^{i-1}a_{l,j}/b_{l} \cdot_{\phi} \overline{\lambda}_{j},\]
where each $a_{l,j}/b_{l} \in A \mod u$ satisfies $b_{l}(a_{l,j}/b_{l}) \equiv a_{l,j} \mod u.$
Hence $\overline{\lambda}_{1},\ldots,\overline{\lambda}_{i-1}$ generate an $i$-dimensional $A/u$-vector space,  which is absurd.

Next, we claim that $\eta_{k}$ and $\lambda_{1},\ldots,\lambda_{i-1}$ are $A/u^{n}$-linearly independent.
Assume the inverse and we have 
\begin{align} c_{k}\cdot_{\phi}\eta_{k} + \sum_{j=1}^{i-1}a_{j}\cdot_{\phi}\lambda_{j} = 0, \label{f21}\end{align}
where each $a_{j}\in A \mod u^{n}$ and $c_{k} \in A \mod u^{n}$ with $c_{k}\not\equiv 0 \mod u^{n}.$   
We may write $c_{k} = c_{k}'u^{m}$ with $m < n$ and $c_{k}' \in A$ not divisible by $u.$ 
Then we have $u^{m}\mid a_{j}$ for all $j < i$, for otherwise, by (\ref{f21}), we have $\sum_{j=1}^{i-1}a_{j}u^{n-m}\cdot_{\phi}\lambda_{j} = 0$ with $a_{j}u^{n-m} \not\equiv 0 \mod u^{n}$ for some $j.$
We may write $a_{j}=a_{j}'u^{m}$ for $a_{j}'\in A.$
Hence we have by (\ref{f21}) \[0 = c_{k}u^{n-1-m}\cdot_{\phi}\eta_{k} + \sum_{j=1}^{i-1}a_{j}u^{n-1-m}\cdot_{\phi}\lambda_{j} = c_{k}'\cdot_{\phi}\overline{\eta}_{k} + \sum_{j=1}^{i-1}a_{j}'\cdot_{\phi}\overline{\lambda}_{j}\] with $c_{k}'\in A$ not divisible by $u.$ 
This contradicts that $\overline{\eta}_{k}$ and $\overline{\lambda}_{1},\ldots, \overline{\lambda}_{i-1}$ are $A/u$-linearly independent.
We have obtained $A/u^{n}$-linearly independent elements $\lambda_{1},\ldots,\lambda_{i-1},\eta_{k}$ such that $w(\eta_{k}) \geq l_{i} > w(\lambda_{i}).$
This contradicts Definition~\ref{d12}~(2).
\end{proof}

In the remainder of this subsection, we construct an SMB of $\phi[u^{n}]$ for any positive integer $n.$

\begin{lemma}\label{l22}
Let $\{\lambda_{i}\}_{i=1,\ldots,r}$ be an SMB of $\phi[u^{n}].$ For each $i$ and $a\in A$ with $a\not\equiv 0\mod u^{n},$ the element $\lambda_{i}$ has the largest valuation among the roots $\lambda$ of $\phi_{a}(X) - a \cdot
_{\phi} \lambda_{i}$ such that $\lambda\in \phi[u^{n}].$
\end{lemma}

\begin{proof}
Let $\lambda$ be a root of $\phi_{a}(X)-a\cdot_{\phi} \lambda_{i}$ such that $\lambda\in \phi[u^{n}].$ 
Assume $w(\lambda)>w(\lambda_{i}).$
It suffices to show that $\lambda_{1},\ldots,\lambda_{i-1},\lambda$ are $A/u^{n}$-linearly independent because this implies that the inequality $w(\lambda)>w(\lambda_{i})$ contradicts Definition~\ref{d12}~(2).
Assume that there exist $b_{j}\in A \mod u^{n}$ with $b_{i}\not\equiv 0$ such that $b_{i}\cdot_{\phi} \lambda+\sum_{j<i}b_{j}\cdot_{\phi} \lambda_{j}=0.$
Let $c$ be the minimal common multiple of $a$ and $b_{i}$ such that $c=b_{i}'b_{i}=a'a$ for some $b_{i}'$ and $a'\in A.$ 
Consider the equation $b_{i}'\cdot_{\phi} (b_{i}\cdot_{\phi} \lambda+\sum_{j<i}b_{j}\cdot_{\phi} \lambda_{j}) = 0.$
Since $b_{i}'b_{i}\cdot_{\phi} \lambda=a'a\cdot_{\phi} \lambda=a'a\cdot_{\phi} \lambda_{i}=c\cdot_{\phi} \lambda_{i},$ we have 
\begin{align}c\cdot_{\phi} \lambda_{i}+\sum_{j<i}b_{i}'b_{j}\cdot_{\phi} \lambda_{j}=0.\label{f22}\end{align}
We have $u^{n} \nmid c,$ for otherwise one of $a$ or $b_{i}$ is divisible by $u^{n}.$
Hence the nonzero coefficients in the equation (\ref{f22}) contradict that $\lambda_{1},\ldots,\lambda_{i}$ are $A/u^{n}$-linearly independent.
\end{proof}

\begin{corollary}\label{c22}
With the notation in the lemma, for each $i$ and $a \in A$ being a power of $u,$ the element $\lambda_{i}$ has the largest valuation among the roots of $\phi_{a}(X)-a\cdot_{\phi} \lambda_{i}.$
\end{corollary}

\begin{proposition}\label{p222}
Let $\{\lambda_{i}\}_{i=1,\ldots,r}$ be an SMB of $\phi[u^{n}].$
\begin{enumerate}[\rm{(}1)]
    \item For each $i$, let $\lambda_{i}'$ be a root of $\phi_{u}(X)-\lambda_{i}$ having the largest valuation.
Then $\{\lambda_{i}'\}_{i=1,\ldots,r}$ is an SMB of $\phi[u^{n+1}].$
\item The family of elements $\{u\cdot_{\phi} \lambda_{i}\}_{i=1,\ldots,r}$ is an SMB of $\phi[u^{n-1}].$
\end{enumerate}
\end{proposition}

\begin{proof}
(1) We check Definition~\ref{d12}~(1) using induction on $i.$
The base case is clear.
Assume $\lambda_{1}',\ldots,\lambda_{i-1}'$ are $A/u^{n+1}$-linearly independent. 
Assume conversely that there are $a_{j}\in A\mod u^{n+1}$ with $a_{i}\not\equiv 0$ such that $\sum_{j=1}^{i}a_{j}\cdot_{\phi} \lambda_{j}'=0.$
For $j=1,\ldots,i,$ since $u\cdot_{\phi} \lambda_{j}'=\lambda_{j}$ and $\lambda_{1},\ldots,\lambda_{i}$ are $A/u^{n}$-linearly independent, we have $ua_{j} \equiv 0 \mod u^{n+1}$ and hence $u^{n}\mid a_{j}.$  
There are $b_{j}\in A$ with $b_{i} \not\equiv 0\mod u$ such that $a_{j}=b_{j}u^{n}$ for all $j.$
Hence
\[0=\sum_{j=1}^{i}a_{i}\cdot_{\phi} \lambda_{i}'=\sum_{j=1}^{i}b_{j}u^{n-1}\cdot_{\phi} \lambda_{i}\] with $b_{i}u^{n-1}$ not divisible by $u^{n},$ which is absurd.

As for Definition~\ref{d12}~(2), we show $w(\lambda_{i}') \geq w(\lambda)$ for each $\lambda\in \phi[u^{n+1}]$ such that $\lambda_{1}',\ldots,\lambda_{i-1}',\lambda$ are $A/u^{n+1}$-linearly independent. 
Notice $u\cdot_{\phi} \lambda\in \phi[u^{n}]$ and that the elements $\lambda_{1},\cdots,\lambda_{i-1},u\cdot_{\phi} \lambda$ are $A/u^{n}$-linearly independent. 
We have $w(\lambda_{i})\geq w(u \cdot_{\phi} \lambda)$ as $\{\lambda_{i}\}_{i=1,\ldots,r}$ is an SMB of $\phi[u^{n}].$
Note that $w(\lambda_{i}')$ is the largest among the valuations of roots of $\phi_{u}(X)-\lambda_{i}.$
By comparing the Newton polygons of $\phi_{u}(X)-\lambda_{i}$ and $\phi_{u}(X)-u\cdot_{\phi}\lambda,$ we have $w(\lambda_{i}')\geq w(\lambda).$ 

(2) It is straightforward to check Definition~\ref{d12}~(1).
Let $\lambda$ be an element of $\phi[u^{n-1}]$ such that $u\cdot_{\phi} \lambda_{1},\ldots,u\cdot_{\phi} \lambda_{i-1},\lambda$ are $A/u^{n-1}$-linearly independent.
To show $w(u\cdot_{\phi} \lambda_{i}) \geq w(\lambda),$ we assume conversely $w(u\cdot_{\phi} \lambda_{i})<w(\lambda).$ 
By comparing the Newton polygon of $\phi_{u}(X) - u\cdot_{\phi} \lambda_{i}$ and $\phi_{u}(X) - \lambda,$ there is a root $\lambda'$ of $\phi_{u}(X)-\lambda$ such that $w(\lambda')>w(\lambda_{i}).$ 
We have $\lambda'\in \phi[u^{n}]$ as all roots of $\phi_{u}(X)-\lambda$ belong to $\phi[u^{n}].$ 
Similarly to the proof of (1), one shows that $\lambda_{1},\cdots,\lambda_{i-1},\lambda'$ are $A/u^{n}$-linearly independent. 
Hence the inequality $w(\lambda')>w(\lambda_{i})$ contradicts Definition~\ref{d12}~(2).
\end{proof}

We can find an SMB of $\phi[u]$ in the following way.
Put 
\begin{equation}\begin{split}
\lambda_{1,1} &\coloneqq\text{an element in }\phi[u]\setminus \{0\}\text{ with the largest valuation and}\\
\lambda_{i,1} &\coloneqq\text{an element in }\phi[u]\setminus \bigoplus_{j<i}(A/u)\cdot_{\phi} \lambda_{j,1}\text{ with the largest valuation} \end{split} \label{f23}\end{equation}
for $i=2,3,\ldots,r.$
Since $A/u$ is a field, the elements $\lambda_{i,1}$ for $i=1,\cdots,r$ are $A/u$-linearly independent and form an SMB of $\phi[u].$
Applying the proposition, we have 

\begin{corollary}\label{c23}
Let $\{\lambda_{i,1}\}_{i=1,\ldots,r}$ be an SMB of $\phi[u]$ defined above.
Inductively, let $\lambda_{i,j}$ be a root of $\phi_{u}(X)-\lambda_{i,j-1}$ having the largest valuation for each $i$ and each integer $j\geq 2.$
Then for each positive integer $n,$ we have that $\{\lambda_{i,n}\}_{i=1,\ldots,r}$ is an SMB of $\phi[u^{n}].$
\end{corollary}

\section{Relations between SMBs, the infinite prime case}\label{s3}
Let $w$ denote an infinite prime, $|-|$ the absolute value in (F\ref{F1}) and $\{\omega_{i}\}_{i=1,\ldots,r}$ an SMB of $\Lambda$ throughout this section.
For a positive integer $n$ and a finite prime $u$ of $A,$ we study the relations between SMBs of $\Lambda$ and those of $\phi[u^{n}].$ 

\begin{lemma}\label{l31}
Let $a$ be an element in $A.$
For $\omega=\sum_{j} a_{j} \omega_{j} \in \Lambda$ with $a_{j}\in A,$ let $i$ be an index so that $|a_{i}\omega_{i}|=|\omega|,$ i.e., $|a_{i}\omega_{i}|=\max_{j}\{|a_{j}\omega_{j}|\}.$ 
Assume $\deg(a_{i})<\deg(a).$
Then we have 
\[\left|e_{\phi}\left(\frac{\omega}{a}\right)\right|=\left|e_{\phi}\left(\frac{a_{i}\omega_{i}}{a}\right)\right|.\]
\end{lemma}

\begin{proof}
We have
\[e_{\phi}\left(\frac{\omega}{a}\right)=\frac{\omega}{a}\prod_{\mu\in \Lambda\setminus\{0\}}\left(1-\frac{\omega}{a\mu}\right).\]
Its absolute value is 
\[\left|\frac{\omega}{a}\right|\cdot \prod_{\substack{\mu\in \Lambda\setminus\{0\}\\|a\mu|\leq |\omega|}}\left|1-\frac{\omega}{a\mu}\right|.\]
For $\mu\in\Lambda$ satisfying $|a\mu|<|\omega|,$ we have by the ultrametric inequality
\[\left|1 - \frac{\omega}{a\mu}\right|=\left|\frac{\omega}{a\mu}\right|=\left|\frac{a_{i}\omega_{i}}{a\mu}\right|=\left|1-\frac{a_{i}\omega_{i}}{a\mu}\right|.\]
Next, for $\mu\in \Lambda$ satisfying $|a\mu|=|\omega|=|a_{i}\omega_{i}|,$ we show 
\[\left|1-\frac{\omega}{a\mu}\right|=\left|1-\frac{a_{i}\omega_{i}}{a\mu}\right|=1.\]
It suffices to show 
\begin{align}|\omega-a\mu|=|\omega|\text{ and }|a_{i}\omega_{i}-a\mu|=|a_{i}\omega_{i}|.\label{f31}\end{align}
Since $|a_{i}|<|a|,$ we have $\mu$ belonging to $\bigoplus_{j<i}A\omega_{j},$ for otherwise we have $|a\mu|\geq |a\omega_{i}|>|a_{i}\omega_{i}|$ by Proposition~\ref{p212}~(2). 
Applying Proposition~\ref{p212}~(2) to $|\omega-a\mu|$ and $|a_{i}\omega_{i}-a\mu|,$ we obtain the desired equalities. 
\end{proof}

\begin{corollary}\label{c31}
Let $a$ be an element in $A.$
\begin{enumerate}[\rm{(}1)]
\item For any $i=1,\ldots,r$ and any $a_{i}\in A$ satisfying $\deg(a_{i}) < \deg(a),$ we have
\[\left| e_{\phi} \left(\frac{a_{i}\omega_{i}}{a}\right) \right|=\left| \frac{a_{i}\omega_{i}}{a} \right|\cdot \prod_{\substack{\mu\in \Lambda\setminus\{0\}\\|a\mu|<|a_{i}\omega_{i}|}} |a_{i}\omega_{i}|/|a\mu| .\]
\item For any positive integers $i,j \leq r,$ let $a_{i}$ and $a_{j}$ be elements in $A$ with degrees strictly smaller than that of $a.$ 
Assume $|a_{j}\omega_{j}|\leq |a_{i}\omega_{i}|.$ Then 
\[\left|e_{\phi}\left(\frac{a_{j}\omega_{j}}{a}\right)\right|\leq \left|e_{\phi}\left(\frac{a_{i}\omega_{i}}{a}\right)\right|.\]
\item With the notation in the lemma,  we have \[\left| e_{\phi} \left(\frac{\omega}{a}\right)\right| =\max_{j} \left\{\left|a_{j}\cdot_{\phi} e_{\phi}\left(\frac{\omega_{j}}{a}\right)\right|\right\}.\]
\item For any positive integer $i \leq r$ and $b\in A$ satisfying $\deg(b)<\deg(a),$ we have \[|b|\cdot \left|e_{\phi}\left(\frac{\omega_{i}}{a}\right)\right|\leq \left|b\cdot_{\phi} e_{\phi}\left(\frac{\omega_{i}}{a}\right)\right|.\]
\end{enumerate}
\end{corollary}

\begin{proof}
(1) has been shown in the proof of the lemma. 
As for (2), by the assumption, we have 
\begin{align}\{\mu \in \Lambda \mid |a\mu|<|a_{j}\omega_{j}|\}\subset \{\mu \in \Lambda \mid |a\mu|<|a_{i}\omega_{i}|\}.\label{f32}\end{align}
If $\mu$ satisfies $|a\mu|<|a_{j}\omega_{j}|,$ we have $|a_{j}\omega_{j}|/|a\mu| \leq |a_{i}\omega_{i}|/|a\mu|.$
Combining this inequality and (\ref{f32}), we have the desired inequality by (1).
For (3), as $a\cdot_{\phi} e_{\phi}(\omega)=e_{\phi}(a\omega)$ for any $a\in A$ and any $\omega\in\Bbb{C}_{\infty},$  it remains to show \[\left| e_{\phi} \left(\frac{\omega}{a}\right)\right| =\max_{j} \left\{\left|e_{\phi}\left(\frac{a_{j}\omega_{j}}{a}\right)\right|\right\}.\]
This equality follows from Lemma~\ref{l31} and (2). 
As for (4), note $|\omega_{i}|<|b\omega_{i}|.$ 
One can show (4) similarly to the proof of (2). 
\end{proof}
\begin{theorem}\label{t31}
For any finite prime $u$ of $A$ and any positive integer $n,$ the family of elements $\{e_{\phi}(\omega_{i}/u^{n})\}_{i=1,\ldots,r}$ is an SMB of $\phi[u^{n}].$
\end{theorem}
\begin{proof}
Put $\lambda_{i}=e_{\phi}(\omega_{i}/u^{n})$ for all $i.$
Note that $\omega_{1}/u^{n},\ldots,\omega_{r}/u^{n}$ are $A/u^{n}$-linearly independent as elements in $u^{-n}\Lambda/\Lambda.$
By the $A/u^{n}$-module isomorphism $\mathcal{E}_{\phi}:u^{-n}\Lambda/\Lambda \to \phi[u^{n}]$ induced by $e_{\phi},$ we have that $\lambda_{1},\ldots,\lambda_{r}$ are $A/u^{n}$-linearly independent.

Fix a positive integer $i \leq r.$ 
To check Definition~\ref{d11}~(2), we show that $|\lambda_{i}|$ is minimal among the absolute values of elements in $\phi[u^{n}]\setminus \bigoplus_{j<i}(A/u^{n})\cdot_{\phi} \lambda_{j}$ (in $\phi[u^{n}] \setminus \{0\}$ if $i=1$).
Put $\lambda=\sum_{j}a_{j}\cdot_{\phi} \lambda_{j}$ with $a_{j}\in A \mod u^{n}$ such that there is $a_{k}\not\equiv 0$ for some $k\geq i.$ 
We show $|\lambda_{i}|\leq |\lambda|.$ 
Without loss of generality, we assume that $\text{}$ 
$\deg(a_{j})<\deg(u^{n})$ for any $j.$ 
Let $l$ be an index so that $|a_{l}\omega_{l}|=|\sum_{j}a_{j}\omega_{j}|.$
By Corollary~\ref{c31}~(3), we have 
\[|\lambda|=\left|a_{l}\cdot_{\phi} \lambda_{l}\right|.\]
As $|a_{k}\omega_{k}| \leq |a_{l}\omega_{l}|,$ Corollary~\ref{c31}~(2) implies 
\[\left|e_{\phi}\left(\frac{a_{k}\omega_{k}}{u^{n}}\right)\right|\leq \left|e_{\phi}\left(\frac{a_{l}\omega_{l}}{u^{n}}\right)\right|,\]
hence $|a_{k}\cdot_{\phi} \lambda_{k}| \leq |a_{l}\cdot_{\phi} \lambda_{l}|.$
As $|\omega_{i}|\leq |\omega_{k}|,$ Corollary~\ref{c31}~(2) also implies $|\lambda_{i}| \leq |\lambda_{k}|.$ 
By Corollary~\ref{c31}~(4), we have $|a_{k}|\cdot |\lambda_{k}|\leq |a_{k}\cdot_{\phi} \lambda_{k}|.$
Combining the equality and inequalities, we have 
\[|\lambda_{i}|\leq |\lambda_{k}|\leq |a_{k}|\cdot |\lambda_{k}|\leq |a_{k}\cdot_{\phi} \lambda_{k}|\leq |a_{l}\cdot_{\phi} \lambda_{l}|=|\lambda|.\]
\end{proof}

\begin{remark}\label{r33}
We have shown in the above proof that $|\lambda_{1}|$ is minimal among the absolute values of nonzero elements in $\phi[u^{n}].$
Let $\{\lambda_{i}l\}_{i=1,\ldots,r}$ be an SMB of $\phi[u^{n}].$
By Theorem~\ref{t32} below, we know that there exists an SMB $\{\omega_{i}'\}_{i=1,\ldots,r}$ on $\Lambda$ such that $e_{\phi}(\omega_{i}'/u^{n}) = \lambda_{i}$ for all $i.$ 
Hence $\lambda_{1}'$ has the minimal absolute value among elements in $\phi[u^{n}]\setminus\{0\}.$
\end{remark}

\begin{corollary}\label{c32}
Let $\{\lambda_{i}\}_{i=1,\ldots,r}$ be an SMB of $\phi[u^{n}].$
\begin{enumerate}[\rm{(}1)]
\item If $n$ is large enough so that $|u^{n}|\geq|\omega_{r}|/|\omega_{1}|,$ then for $i=1,\ldots,r,$ we have $|\lambda_{i}|\cdot |u^{n}|=|\omega_{i}|.$
\item For any positive integer $n,$ we have $|\lambda_{r}|/|\lambda_{1}|\geq |\omega_{r}|/|\omega_{1}|.$
\item If $n$ is large enough such that  $|u^{n}|>|\omega_{r}|/|\omega_{1}|,$ then we have $|\lambda_{i}|<|\omega_{1}|$ for $i=1,\ldots,r.$
\end{enumerate}
\end{corollary}

\begin{proof}
We show (1).
Fix $i$ to be one of $1,\ldots,r.$
Corollary~\ref{c31}~(1) implies
\begin{align}
\left| e_{\phi}\left(\frac{\omega_{i}}{u^{n}}\right)\right|=\left|\frac{\omega_{i}}{u^{n}}\right|\cdot \prod_{\substack{\mu\in \Lambda\setminus\{0\}\\
|u^{n}\mu|<|\omega_{i}|}} |\omega_{i}|/|u^{n}\mu|. \label{f33}\end{align}
For any $\mu\in \Lambda,$ we have \[|u^{n}\mu|\geq |u^{n}\omega_{1}|\geq |\omega_{r}|\geq |\omega_{i}|\] by the hypothesis.
Hence (\ref{f33}) implies
\[\left| e_{\phi}\left(\frac{\omega_{i}}{u^{n}}\right)\right|=\left|\frac{\omega_{i}}{u^{n}}\right|.\]
By Theorem~\ref{t31}, the family $\{e_{\phi}(\omega_{i}/u^{n})\}_{i=1,\ldots,r}$ is an SMB of $\phi[u^{n}].$
 Hence we have \begin{align}|\lambda_{i}|=\left|e_{\phi}\left(\frac{\omega_{i}}{u^{n}}\right)\right|\text{ for any }i\label{f34}\end{align} by Proposition~\ref{p221}~(2). (1) follows.
Notice that (\ref{f33}) implies
\[\left|e_{\phi}\left(\frac{\omega_{1}}{u^{n}}\right)\right|=\left|\frac{\omega_{1}}{u^{n}}\right|\text{ and }\left|e_{\phi}\left(\frac{\omega_{i}}{u^{n}}\right)\right|\geq \left|\frac{\omega_{i}}{u^{n}}\right|\text{ for any }i.\]
(2) follows from (\ref{f34}).
Since we know $|\lambda_{r}|=|\omega_{r}|/|u^{n}|$ by (1), we have
\[|\lambda_{i}|\leq |\lambda_{r}|=|\omega_{r}|/|u^{n}|<|\omega_{r}|/(|\omega_{r}|/|\omega_{1}|)=|\omega_{1}|\]
and (3) follows.
\end{proof}

\begin{remark}\label{r31}
By Corollary~\ref{c32}~(1) and (2), we have $|\lambda_{i}|\cdot |u^{n}|=|\omega_{i}|$ if $n$ is large enough so that $|u^{n}| \geq |\lambda_{r}|/|\lambda_{1}|.$
\end{remark}

Put  $B:=\{ \omega \in \Bbb{C}_{\infty} \mid |\omega|<|\omega_{1}|\}.$
Since $B\cap \Lambda = \varnothing,$ the exponential function $e_{\phi}$ is injective on $B.$
For any $\omega\in \Bbb{C}_{\infty},$ we have 
\begin{align}|e_{\phi}(\omega)|=|\omega|\cdot \prod_{\substack{\mu\in \Lambda\setminus\{0\}\\|\mu|\leq |\omega|}}\left|1-\frac{\omega}{\mu}\right|.\label{f35}\end{align}
Hence $|e_{\phi}(\omega)| = |\omega|$ for $\omega \in B.$  
This implies $e_{\phi}(B)\subset B.$
Put $C \coloneqq e_{\phi}(B).$ 
There is an inverse $\log_{\phi}:C\to B$ of $e_{\phi}$ defined by a power series with coefficients in $K$ and $e_{\phi}:B\rightleftarrows C:\log_{\phi}$ are inverse to each other.

\begin{lemma}\label{l32}
\begin{enumerate}[\rm{(}1)]
\item We have $C = B.$ 
\item We have the following maps which are inverse to each other
\[e_{\phi}:B\cap \mathcal{L}\rightleftarrows B\cap \phi[u^{n}]:\log_{\phi},\]
where \[\mathcal{L}:=\bigg\{ \sum_{i} a_{i}(\omega_{i}/u^{n}) \,\bigg|\,a_{i}\in A\text{ with }\deg(a_{i})<\deg(u^{n})\bigg\}\]
is a set of representatives of all elements in $u^{-n}\Lambda/\Lambda.$
\item For any $\lambda\in B\cap \phi[u^{n}],$ we have $|\log_{\phi}(\lambda)|=|\lambda|.$
\end{enumerate}
\end{lemma}

\begin{proof}
We show (1) using a property of the image of the open disk $B$ under the power series $e_{\phi}.$
Let $c_{i}$ be elements in $\Bbb{C}_{\infty}$ so that 
\[\sum_{i \geq 1} c_{i}\omega^{i} \coloneqq \omega \prod_{\mu\in\Lambda\setminus\{0\}}(1-\omega/\mu) = e_{\phi}(\omega).\]
We first calculate the minimal integer $d$ such that $|c_{d}||\omega_{1}^{d}|$ is the maximal among $|c_{i}||\omega_{1}^{i}|$ for all $i,$ i.e., $d$ is the Weiestrass degree of $e_{\phi}$ on $B.$ 
Clearly $c_{1} = 1.$
As $|\omega_{1}/\mu| \leq 1$ for any $\mu \in \Lambda\setminus\{0\},$ we have the following inequalities for each integer $i \geq 2$
\[|c_{i}||\omega_{1}^{i}| \leq \sup_{\mu_{j}\in\Lambda\setminus\{0\}}\bigg\{|\omega_{1}|\cdot\bigg|\prod_{j=1}^{i-1}\omega_{1}/\mu_{j}\bigg|\bigg\} \leq |\omega_{1}| = |c_{1}||\omega_{1}|.\]
Hence $d=1.$
By \cite[Theorem~3.15]{Ben}, we have $C = e_{\phi}(B) = B.$

As we have bijections $e_{\phi}:B \to B$ and $e_{\phi}:\mathcal{L} \to \phi[u^{n}],$ (2) follows.
As for (3), by (2), we have $\log_{\phi}(\lambda)\in B\cap \mathcal{L}$ and $e_{\phi}(\log_{\phi}(\lambda))=\lambda.$
Hence we have $|\log_{\phi}(\lambda)|=|\lambda|$ by (\ref{f35}).
\end{proof}

Let $\{\lambda_{i}\}_{i=1,\ldots,r}$ denote an SMB of $\phi[u^{n}].$ 
Assume that the positive integer $n$ is large enough so that $|u^{n}|>|\omega_{r}|/|\omega_{1}|.$ 
By Corollary~\ref{c32}~(3) and Lemma~\ref{l32}~(1), for each $i,$ we have $\lambda_{i} \in B \cap \phi[u^{n}] = C\cap \phi[u^{n}]$ and we put $\omega_{i}'\coloneqq \log_{\phi}(\lambda_{i}).$

\begin{theorem}\label{t32}
The family $\{u^{n}\omega_{i}'\}_{i=1,\ldots,r}$ is an SMB of $\Lambda.$
\end{theorem}

We need a lemma in the proof.

\begin{lemma}\label{l33}
Let $\{\eta_{i}\}_{i=1,\ldots,r}$ be a family of elements in $u^{-n}\Lambda.$
It is an SMB of $u^{-n}\Lambda$ if and only if $\{u^{n}\eta_{i}\}_{i=1,\ldots,r}$ is an SMB of $\Lambda.$
\end{lemma}

\begin{proof}[Proof of Lemma]
For any $a_{i} \in A,$ we have 
\[\bigg|\sum_{i}a_{i}u^{n}\eta_{i}\bigg|=\left|u^{n}\right|\cdot \bigg|\sum_{i}a_{i}\eta_{i}\bigg|.\]
Then the lemma follows from Proposition~\ref{p212}.
\end{proof}

\begin{proof}[Proof of Theorem]
By Lemma~\ref{l33}, it suffices to show that the family of elements $\{\omega_{i}'\}_{i=1,\ldots,r}$ is an SMB of $u^{-n}\Lambda.$ 
To check the first dot in Proposition~\ref{p211}~(1), we show that $\omega_{1}',\ldots,\omega_{r}'$ are $A$-linearly independent.
Assume that there exist nonzero $a_{i}\in A$ such that $\sum_{i}a_{i}\omega_{i}'=0.$
We may assume $u^{n} \nmid a_{i}$ for some $i,$ for otherwise we divide both sides of the equation $\sum_{i}a_{i}\omega_{i}'=0$ by some power of $u.$
Note that the map $e_{\phi}$ is $A/u^{n}$-linear. 
As some $a_{i}$ satisfies $a_{i} \not\equiv 0 \mod u^{n}$ and $\lambda_{1},\ldots,\lambda_{r}$ are $A/u^{n}$-linearly independent, we have $e_{\phi}(\sum_{i}a_{i}\omega_{i}')=\sum_{i}a_{i}\cdot_{\phi} \lambda_{i}\neq 0.$
This is absurd.

Next, we check the second dot in Proposition~\ref{p211}~(1).
Let $l_{1} \leq l_{2} \leq \cdots \leq l_{r}$ be the invariant of $u^{-n}\Lambda$ as in Proposition~\ref{p211}~(2). 
Fix $i$ to be a positive integer $\leq r.$
It suffices to show $l_{i}=|\omega_{i}'|.$
We have $l_{i}\leq |\omega_{i}'|.$ 
Let us assume $l_{i}<|\omega_{i}'|.$
As $\lambda_{i}\in B\cap \phi[u^{n}],$ we have $|\omega_{i}'|=|\lambda_{i}|$ by Lemma~\ref{l32}~(3).
Hence $l_{i}<|\omega_{i}'|=|\lambda_{i}|<|\omega_{1}|.$
By Proposition~\ref{p211}~(1), there is an SMB $\{\eta_{j}\}_{j=1,\ldots,r}$ of $u^{-n}\Lambda$ such that $|\eta_{i}|=l_{i}<|\omega_{1}|.$
As $|\eta_{i}|<|\omega_{1}|,$ we know $|e_{\phi}(\eta_{i})|=|\eta_{i}|$ from (\ref{f35}).
We have 
\[|e_{\phi}(\eta_{i})|=|\eta_{i}|=l_{i}<|\omega_{i}'|=|\lambda_{i}|\]
and hence $|e_{\phi}(\eta_{i})|<|\lambda_{i}|.$
On the other hand, note that $\{u^{n}\eta_{j}\}_{j=1,\ldots,r}$ is an SMB of $\Lambda$ by Lemma~\ref{l33}.
By Theorem~\ref{t31}, the elements $e_{\phi}(\eta_{j})$ for $j=1,\ldots,r$ form an SMB of $\phi[u^{n}].$
By Proposition~\ref{p221}~(2), this contradicts $|e_{\phi}(\eta_{i})|<|\lambda_{i}|.$
\end{proof}

Finally, we give two applications of Theorem~\ref{t31} and \ref{t32}.

\begin{proposition}\label{p31}
If $n$ is large enough so that $|u^{n}|>|\omega_{r}|/|\omega_{1}|,$ then we have 
\[K(\Lambda)=K(\phi[u^{n}]),\]
where $K(\Lambda)$ \text{\rm{(}}resp. $K(\phi[u^{n}])$\text{\rm{)}} is the extension of $K$ generated by all elements in $\Lambda$ \text{\rm{(}}resp. in $\phi[u^{n}]$\text{\rm{)}}.
\end{proposition}

\begin{proof}
(cf. the proof of \cite[Proposition~2.1]{Mau}) Note that $e_{\phi}$ is given by a power series with coefficients in $K.$  
For any $x\in K^{\mathrm{sep}},$ we have $e_{\phi}(x)\in K(x)$ since the field $K(x)$ is complete.
Since the exponential map $e_{\phi}$ induces a bijection $u^{-n}\Lambda/\Lambda \to \phi[u^{n}]$, for any $\lambda$ in $\phi[u^{n}],$ there exists $\omega\in u^{-n}\Lambda$ such that $e_{\phi}(\omega)=\lambda.$
This implies $K(\lambda)\subset K(\omega)$ and $K(\phi[u^{n}])\subset K(\Lambda).$

Note that $\log_{\phi}$ is given by a power series with coefficients in $K.$
For any $y\in c\cap K^{\mathrm{sep}},$ we similarly have $\log_{\phi}(y)\in K(y).$ 
Let $\{\lambda_{i}\}_{i=1,\ldots,r}$ be an SMB of $\phi[u^{n}].$ 
As $|u^{n}|>|\omega_{r}|/|\omega_{1}|,$ by Theorem~\ref{t32}, the elements $u^{n}\omega_{i}'$ for $i=1,\ldots,r$ form an SMB of $\Lambda,$ where $\omega_{i}'=\log_{\phi}(\lambda_{i}).$
Since $K(\omega_{i}')\subset K(\lambda_{i})$ for each $i,$ we have $K(\Lambda)\subset K(\phi[u^{n}]).$
\end{proof}

Combining Corollary~\ref{c32}~(2), Theorem~\ref{t32}, and Proposition~\ref{p31}, we have 
\begin{corollary}\label{r32}
Let $l$ be a positive integer and $\{\eta_{i}\}_{i=1,\ldots,r}$ an SMB of $\phi[u^{l}].$
Let $\{\lambda_{i}\}_{i=1,\ldots,r}$ be an SMB of $\phi[u^{n}].$
If $n$ is large enough so that $|u^{n}| > |\eta_{r}|/|\eta_{1}|,$ then we have
\begin{enumerate}[\rm{(}1)]
\item the family $\{u^{n}\log_{\phi}(\lambda_{i})\}_{i=1,\ldots,r}$ is an SMB of $\Lambda;$
\item $K(\Lambda)=K(\phi[u^{n}]).$
\end{enumerate}
\end{corollary}

\begin{proposition}\label{p32}
Let $\{\lambda_{i}\}_{i=1,\ldots,r}$ be an SMB of $\phi[u^{n}].$ 
We have
\[\bigg|\sum_{i}a_{i}\cdot_{\phi} \lambda_{i}\bigg|=\max_{i}\{|a_{i}\cdot_{\phi} \lambda_{i}|\}\]
for any $a_{i}\in A\mod u^{n}.$
\end{proposition}

\begin{proof}
Without loss of generality, we assume $\deg(a_{i}) < \deg(u^{n})$ for all $i.$
Assume first that $n$ is large enough so that $|u^{n}|>|\lambda_{r}|/|\lambda_{1}|$ (Corollary~\ref{c32}~(2)).  
By Theorem~\ref{t32}, the elements $u^{n}\omega_{i}'$ for $i=1,\ldots,r$ form an SMB of $\Lambda,$ where $\omega_{i}'=\log_{\phi}(\lambda_{i}).$
By Corollary~\ref{c31}~(3), we have 
\[\bigg|e_{\phi}\bigg(\sum_{i}a_{i}\omega_{i}'\bigg)\bigg|=\max_{i}\{|a_{i}\cdot_{\phi} e_{\phi}(\omega_{i}')|\}.\]
As  $e_{\phi}(\sum_{i}a_{i}\omega_{i}')=\sum_{i}a_{i}\cdot_{\phi} \lambda_{i},$ the claim follows.

For any $n,$ let $n'$ be an integer $\geq n$ so that $|u^{n'}|>|\lambda_{r}|/|\lambda_{1}|.$
By Proposition~\ref{p222}~(1), there is an SMB $\{\lambda_{i}'\}_{i=1,\ldots,r}$ of $\phi[u^{n'}]$ such that $u^{n'-n}\cdot_{\phi} \lambda_{i}'=\lambda_{i}$ for all $i.$
Then the desired equation for $\{\lambda_{i}\}_{i=1,\ldots,r}$ follows from that for $\{\lambda_{i}'\}_{i=1,\ldots,r}.$
\end{proof}

\begin{corollary}[cf. Lemma~\ref{l21}]\label{l53}
Let $\{\lambda_{i}\}_{i=1,\ldots,r}$ be an SMB of $\phi[u^{n}]$ such that $|\lambda_{1}| = \cdots = |\lambda_{s}| < |\lambda_{s+1}|$ for some positive integer $s.$  
\begin{enumerate}[\rm{(}1)]
\item
The extension of $K$ generated by $\lambda_{1},\ldots,\lambda_{s}$ is at worst tamely ramified. 
\item
For an element $\sigma$ in the wild ramification subgroup $\mathrm{Gal}(K(\phi[u^{n}])/K)_{1},$ we have $\sigma(\lambda_{j}) = \lambda_{j}$ for $j=1,\ldots,s.$ 
\end{enumerate}
\end{corollary}

\begin{proof}
(2) straightforwardly follows from (1).
We show (1).
Let $n'$ be an integer satisfying $n'\geq n$ and $|u^{n'}| > |\lambda_{r}|/|\lambda_{1}|.$
By Proposition~\ref{p222}~(1), we know that there exists an SMB $\{\lambda_{i}'\}_{i=1,\ldots,r}$ of $\phi[u^{n'}]$ such that $u^{n'-n}\cdot_{\phi}\lambda_{i}' = \lambda_{i}$ for all $i.$
By Corollary~\ref{c22}, we have $|\lambda_{1}'| = \cdots = |\lambda_{s}'| < |\lambda_{s+1}'|.$
By Theorem~\ref{t32}, we know that there exists an SMB $\{\omega_{i}\}_{i=1,\ldots,r}$ of $\Lambda$ such that $e_{\phi}(\omega_{i}/u^{n'}) = \lambda_{i}'$ for all $i.$ 
Corollary~\ref{c32} implies that $|\omega_{1}| = \cdots = |\omega_{s}| < |\omega_{s+1}|.$
As $e_{\phi}(\omega_{i}/u^{n}) = \lambda_{i},$ we have $K(\lambda_{i}) \subset K(\omega_{i})$ for all $i.$
Then the result follows from Lemma~\ref{l21}.
\end{proof}

\section{Relations between SMBs, the finite prime case}\label{s4}
Throughout this section, let $w$ denote a finite prime and assume that $\phi$ has stable reduction over $K.$
Let $\{\omega_{i}^{0}\}_{i=r'+1,\ldots,r}$ be an SMB of $\Lambda.$ 
Let $|-|$ denote the function in (F\ref{F2}) and put $|a|_{\infty} \coloneqq q^{\deg(a)}$ for any $a\in A.$ 
For a positive integer $n$ and a finite prime $u$ of $A,$ we study the relations between SMBs of $\psi[u^{n}]$, those of $\Lambda,$ and those of $\phi[u^{n}].$

First, we are concerned with the valuations of the elements in the $A$-module $u^{-n}\Lambda,$ i.e., the roots of $\psi_{u^{n}}(X)-\omega$ for all $\omega\in \Lambda.$

\begin{lemma}\label{l41}
Let $a$ be an element in $A.$
\begin{enumerate}[\rm{(}1)]
\item Each root of $\psi_{a}(X)$ has valuation $\geq 0.$ 
Moreover, all nonzero roots of $\psi_{a}(X)$ have valuation $=0$ if and only if $w(a) = 0.$
\item For a nonzero element $\omega\in \Lambda,$ each root of $\psi_{a}(X)-\omega$ has valuation $<0.$
\item An element $\omega\in a^{-1}\Lambda$ belongs to $\psi[a]$ if and only if it has valuation $\geq 0.$
\end{enumerate}
\end{lemma}

\begin{proof}
Put $g \coloneqq r'\cdot \deg(a),$ $a_{0} \coloneqq a,$ $\sum_{i=0}^{g}a_{i}X^{q^{i}} \coloneqq \psi_{a}(X),$ and $P_{i}=(q^{i},w(a_{i}))$ for $i=0,\ldots,g.$
As $w(a_{i})\geq 0$ and $w(a_{g})=0,$ the segments in the Newton polygon of $\psi_{a}(X)$ have slopes $\leq 0.$
If $w(a_{0})=0,$ then the Newton polygon of $\psi_{a}(X)$ consists of exactly one segment $P_{0}P_{g}$ which has slope $0.$
Hence each root of $\psi_{a}(X)$ has valuation $=0.$
If $w(a_{0}) > 0,$ then the left-most segment in the Newton polygon of $\psi_{a}(X)$ has negative slope. 
Hence some root of $\psi_{a}(X)$ has valuation $>0.$

As for (2), put $Q\coloneqq (0,w(\omega)).$ 
As $w(\omega)<0,$ $w(a_{i})\geq 0$ for all $i,$ and $w(a_{g})=0,$ the Newton polygon of $\psi_{a}(X)-\omega$ consists of exactly one segment $QP_{g}$ whose slope is $-w(\omega)/q^{g}>0.$
Hence (2) follows.
From (1) and (2), we know (3).
\end{proof}

Fix a root $\omega_{i}$ of $\psi_{u^{n}}(X)-\omega_{i}^{0}$ for $i=r'+1,\ldots,r.$
The elements $\omega_{r'+1},\ldots,\omega_{r}$ are $A$-linearly independent.
For all $a_{i}\in A,$ we have  
\[|u^{n}|_{\infty}\cdot \bigg|\sum_{i=r'+1}^{r}a_{i}\cdot_{\psi} \omega_{i}\bigg| = \bigg|\sum_{i=r'+1}^{r}a_{i}u^{n} \cdot_{\psi} \omega_{i}\bigg| = \bigg|\sum_{i=r'+1}^{r}a_{i} \cdot_{\psi} \omega_{i}^{0}\bigg|.\]
Hence, by Proposition~\ref{p212}, we have \begin{align}\bigg|\sum_{i=r'+1}^{r}a_{i} \cdot_{\psi} \omega_{i}\bigg|=\max_{i=r'+1,\ldots,r}\{|a_{i} \cdot_{\psi} \omega_{i}|\}\label{f45}\end{align} for any $a_{i}\in A.$

In the remainder of this section, let $\{\omega_{i}\}_{i=1,\ldots,r'}$ be an SMB of $\psi[u^{n}]$ and $\omega_{r'+1},\ldots,\omega_{r}$ be elements in $u^{-n}\Lambda$ defined as above.
The family $\{\omega_{i}\}_{i=1,\ldots,r}$ form an $A/u^{n}$-basis of $u^{-n}\Lambda/\Lambda.$ 
Next, we study the relations between $\{\omega_{i}\}_{i=1,\ldots,r}$ and SMBs of $\phi[u^{n}].$

\begin{lemma}\label{l43}
\begin{enumerate}[\rm{(}1)] 
\item For all $a_{i}\in A,$ we have
\[\bigg|\sum_{i}a_{i} \cdot_{\psi} \omega_{i}\bigg| = \begin{cases}|\sum_{i\leq r'}a_{i} \cdot_{\psi} \omega_{i}| \leq 0 & \text{all }a_{i}=0\text{ for }i> r';\\
|\sum_{i>r'}a_{i} \cdot_{\psi} \omega_{i}|>0 & \text{some }a_{i}\neq 0\text{ for }i> r'.\end{cases}\]
\item   Let $a_{i}$ be elements in $A$ for $i=1,\ldots,r.$ 
Assume either $w(u)=0,$ or some $a_{i}$ is nonzero for $i>r'.$ Then we have 
\[\bigg|\sum_{i}a_{i} \cdot_{\psi} \omega_{i}\bigg| = \max_{i}\{|a_{i} \cdot_{\psi} \omega_{i}|\}.\]
\end{enumerate}
\end{lemma}

\begin{proof}
(1) Since $\sum_{i\leq r'}a_{i} \cdot_{\psi} \omega_{i}\in \psi[u^{n}],$ we have $|\sum_{i\leq r'}a_{i} \cdot_{\psi} \omega_{i}| \leq 0$ by Lemma~\ref{l41}~(3). 
Since $u^{n} \cdot_{\psi} \omega_{i}$ for all $i=r'+1,\ldots,r$ are elements in $\Lambda,$ we have $|u^{n}|_{\infty}\cdot |\omega_{i}|>0$ and hence $|a_{i}|_{\infty}\cdot |\omega_{i}|>0$ if $a_{i}$ is nonzero.
Hence, by (\ref{f45}) and the ultrametric inequality, we have $|\sum_{i}a_{i} \cdot_{\psi} \omega_{i}|=|\sum_{i>r'}a_{i} \cdot_{\psi} \omega_{i}|>0$ if some $a_{i}$ for $i > r'$ is nonzero. 
(1) follows.

(2) If some $a_{i}\neq 0$ for $i>r',$ the desired equality follows from (1) and (\ref{f45}). 
By Lemma~\ref{l41}~(1), the assumption $w(u)=0$ implies that the elements in $\psi[u^{n}]$ have valuation $0.$
Hence $|\sum_{i\leq r'}a_{i} \cdot_{\psi} \omega_{i}|=0$ and $|a_{i} \cdot_{\psi} \omega_{i}| = 0$ for all $i \leq r'.$ 
The desired equality similarly follows.
\end{proof}

Recall for any $\omega\in \Bbb{C}_{w},$ we have
\[e_{\phi}(\omega)=\omega\prod_{\mu\in \Lambda\setminus\{0\}}\left(1-\frac{\omega}{\mu}\right).\]
Its valuation is 
\begin{align}w(e_{\phi}(\omega))=w(\omega)+\sum_{\substack{\mu\in \Lambda\setminus\{0\}\\
w(\mu)\geq w(\omega)}}w\left(1-\frac{\omega}{\mu}\right).\label{f41}\end{align}
For certain $\omega=\sum_{i}a_{i} \cdot_{\psi} \omega_{i}\in u^{-n}\Lambda,$ we are to calculate $|e_{\phi}(\omega)|.$
\begin{lemma}\label{l44}
 If $\omega = \sum_{i\leq r'}a_{i} \cdot_{\psi} \omega_{i}$ with $a_{i}\in A\mod u^{n},$ we have 
\[|e_{\phi}(\omega)|=|\omega|.\]
\end{lemma}

\begin{proof}
By (\ref{f41}), it suffices to show $w(1-\omega/\mu)=0$ for each $\mu\in \Lambda.$
Notice $w(\omega)\geq 0$ by Lemma~\ref{l43}~(1).
Since $w(\mu)<0$ for any $\mu\in \Lambda,$ we have $w(1-\omega/\mu)=0$ by the ultrametric inequality.
\end{proof}

\begin{lemma}[\text{cf. Lemma~\ref{l31}}]\label{l45}
For $\omega=\sum_{j}a_{j} \cdot_{\psi} \omega_{j} \in u^{-n}\Lambda,$ assume some $a_{j}$ for $j>r'$ is nonzero. 
Let $i$ be an integer $>r'$ such that $|\omega|=|a_{i} \cdot_{\psi} \omega_{i}|=\max_{j}\{|a_{j} \cdot_{\psi} \omega_{j}|\}$ \text{\rm{(}}by Lemma~\text{\rm{\ref{l43}~(2)}}\text{\rm{)}}.
Assume $\deg(a_{i})<\deg(u^{n}).$
Then we have 
\[|e_{\phi}(\omega)|=|e_{\phi}(a_{i} \cdot_{\psi} \omega_{i})|.\]
\end{lemma}

\begin{proof}
By (\ref{f41}), it suffices  to show \[w\left(1-\frac{\omega}{\mu}\right) = w\left(1-\frac{a_{i} \cdot_{\psi} \omega_{i}}{\mu}\right)\] for each $\mu\in \Lambda$ with $w(\mu)\geq w(\omega).$
If $w(\mu)>w(\omega),$ then we have by the ultrametric inequality that
\[w\left(1-\frac{\omega}{\mu}\right) = w\left(\frac{\omega}{\mu}\right) = w\left(\frac{a_{i} \cdot_{\psi} \omega_{i}}{\mu}\right) = w\left(1-\frac{a_{i} \cdot_{\psi} \omega_{i}}{\mu}\right).\]
Next, we show 
\[w \left( 1-\frac{\omega}{\mu} \right) = w \left( 1-\frac{a_{i} \cdot_{\psi} \omega_{i}}{\mu}\right)=0\]
if $w(\mu)=w(\omega)=w(a_{i} \cdot_{\psi} \omega_{i}).$
It suffices to show 
\[w( \omega-\mu ) = w(\omega)\text{ and }w( a_{i} \cdot_{\psi} \omega_{i}-\mu ) = w( a_{i} \cdot_{\psi} \omega_{i} ).\]
As $\deg(a_{i})<\deg(u^{n}),$ we have 
\[|\omega| = |a_{i} \cdot_{\psi} \omega_{i}| = |a_{i}|_{\infty}\cdot |\omega_{i}| < |u^{n}|_{\infty}\cdot |\omega_{i}| = |\omega_{i}^{0}|\]
and hence $|\mu| = |\omega| < |\omega_{i}^{0}|.$
This implies $\mu\in \bigoplus_{j=r'+1}^{i-1}A \cdot_{\psi} \omega_{j}^{0},$ for otherwise we have  
$|\mu|\geq |\omega_{i}^{0}|$ by Proposition~\ref{p212}~(2).
Applying Lemma~\ref{l43}~(2) to $|\omega-\mu|$ and $|a_{i} \cdot_{\psi} \omega_{i}-\mu|,$ we obtain the desired equalities.
\end{proof}

\begin{corollary}[\text{cf. Corollary~\ref{c31}}]\label{c41}
\begin{enumerate}[\rm{(}1)]
\item  With the notation in the Lemma~\text{\rm{\ref{l45}}}, we have
\[w(e_{\phi}(\omega)) = w(\omega) + \sum_{\substack{\mu\in\Lambda\setminus\{0\}\\w(\mu)>w(\omega)}}\left(w(\omega) - w(\mu)\right).\]
Particularly, for any $i=1,\ldots,r$ and any $a_{i}\in A\setminus\{0\}$ satisfying $\deg(a_{i})<\deg(u^{n}),$ we have 
\[w(e_{\phi}(a_{i} \cdot_{\psi} \omega_{i})) = w(a_{i} \cdot_{\psi} \omega_{i}) + \sum_{\substack{\mu\in \Lambda\setminus\{0\}\\w(\mu)>w(a_{i} \cdot_{\psi} \omega_{i})}}\big(w(a_{i} \cdot_{\psi} \omega_{i})-w(\mu)\big).\]
\item For any positive integers $i,j \leq r,$ let $a_{i}$ and $a_{j}$ be elements in $A$ with degree strictly smaller than that of $a.$
Assume $|a_{j} \cdot_{\psi} \omega_{j}|\leq |a_{i} \cdot_{\psi} \omega_{i}|.$ 
Then 
\[|e_{\phi}(a_{j} \cdot_{\psi} \omega_{j})|\leq|e_{\phi}(a_{i} \cdot_{\psi} \omega_{i})|.\]
\item With the notation in the lemma, we have
\[|e_{\phi}(\omega)|=\max_{j}\{|a_{j}\cdot_{\phi} e_{\phi}(\omega_{j})|\}.\]\
\item For any positive integer $i = r'+1,\ldots,r$ and $b\in A$ satisfying $\deg(b)<\deg(a),$ we have \[|b|_{\infty}\cdot |e_{\phi}(\omega_{i})|\leq |b\cdot_{\phi} e_{\phi}(\omega_{i})|.\]
\end{enumerate}
\end{corollary}

\begin{proof}
If $i\leq r',$ then we have $w(e_{\phi}(a_{i}\cdot_{\psi}\omega_{i})) = w(a_{i}\cdot_{\psi}\omega_{i})$ by Lemma~\ref{l44}.
The rest of (1) has been shown in the lemma.
Similarly to the proof of Corollary~\ref{c31}~(2) (resp. (3)), the claim (2) (resp. (3)) follows from (1) (resp. the lemma and (2)). 

We show (4).
Note $b\cdot_{\phi} e_{\phi}(\omega_{i})=e_{\phi}(b \cdot_{\psi} \omega_{i}).$
By (1), the desired inequality in (4) is equivalent to 
\begin{equation}\begin{split}&|b|_{\infty}^{r'} \cdot \bigg(w(\omega_{i})+\sum_{\substack{\mu\in \Lambda\setminus\{0\}\\ w(\mu)>w(\omega_{i})}}\big(w(\omega_{i})-w(\mu)\big)\bigg)\\
&\geq w(b \cdot_{\psi} \omega_{i}) + \sum_{\substack{\mu\in \Lambda\setminus\{0\}\\w(\mu)>w(b \cdot_{\psi} \omega_{i})}}\big(w(b \cdot_{\psi} \omega_{i})-w(\mu)\big).
\end{split}\label{f42}\end{equation}
By Lemma~\ref{l20}~(2), we may write the left in this inequality to be 
\[w(b \cdot_{\psi} \omega_{i})+\sum_{\substack{\mu\in\Lambda\setminus\{0\}\\w(\mu)>w(\omega_{i})}}\big(w(b \cdot_{\psi} \omega_{i})-w(b \cdot_{\psi} \mu)\big).\]
Then (\ref{f42}) follows from the inclusion
\[\{ b \cdot_{\psi} \mu \in b \cdot_{\psi} \Lambda \mid w(b \cdot_{\psi} \mu)>w(b \cdot_{\psi} \omega_{i}) \} \subset \{ \mu\in \Lambda \mid w(\mu)>w(b \cdot_{\psi} \omega_{i}) \}.\]
\end{proof}

\begin{theorem}[\text{cf. Theorem~\ref{t31}}]\label{t41}
For any finite prime $u$ of $A$ and any positive integer $n,$ let $\{\omega_{i}\}_{i=1,\ldots,r}$ be the elements in $u^{-n}\Lambda$ defined before Lemma~\text{\rm{\ref{l43}}}.
Then the family of elements $\{e_{\phi}(\omega_{i})\}_{i=1,\ldots,r}$ is an SMB of $\phi[u^{n}].$
\end{theorem}

\begin{proof}
Put $\lambda_{i} \coloneqq e_{\phi}(\omega_{i})$ for all $i.$
Since $\omega_{1},\ldots,\omega_{r}$ form an $A/u^{n}$-basis of $u^{-n}\Lambda/\Lambda,$ their images under the $A/u^{n}$-module isomorphism $\mathcal{E}_{\phi}: u^{-n}\Lambda/\Lambda\to \phi[u^{n}]$ are $A/u^{n}$-linearly independent.

We check Definition~\ref{d11}~(2).
Fix a positive integer $i \leq r.$
For $\lambda=\sum_{j}a_{j}\cdot_{\phi} \lambda_{j}$ with $a_{j}\in A\mod u^{n}$ such that $\lambda_{1},\ldots,\lambda_{i-1},\lambda$ are $A/u^{n}$-linearly independent, we show $|\lambda_{i}|\leq |\lambda|.$
Without loss of generality, we assume $\text{}$ 
$\deg(a_{j})<\deg(u^{n})$ for any $j.$ 

Assume first $i\leq r'.$ 
If $a_{j} = 0$ for all $j>r'$, the desired inequality follows from $\{\omega_{j}\}_{j=1,\ldots,r'}$ being an SMB of $\psi[u^{n}]$ and Lemma~\ref{l44}.
If $a_{j} \neq 0\mod u^{n}$ for some $j>r',$ we can apply Corollary~\ref{c41}~(1), and we have $|\sum_{j} a_{j} \cdot_{\psi} \omega_{j}| \leq |\sum_{j}a_{j} \cdot_{\phi} \lambda_{j}|.$
We know $|\sum_{j} a_{j} \cdot_{\psi} \omega_{j}| \geq 0$ from Lemma~\ref{l43}~(1).
By Lemma~\ref{l43}~(1) and \ref{l44}, we have  $|\lambda_{i}|=|\omega_{i}|<0.$
Hence
\[|\lambda_{i}| = |\omega_{i}| < 0 \leq \bigg|\sum_{j} a_{j} \cdot_{\psi} \omega_{j}\bigg| = \bigg|\sum_{j}a_{j} \cdot_{\phi} \lambda_{j}\bigg|.\]

As for the case $i\geq r'+1,$ note that there is $a_{k} \neq 0$ for some $k\geq i$ as $\lambda_{1},\ldots,\lambda_{i-1},\lambda$  are $A/u^{n}$-linearly independent.
Similarly to  the proof of Theorem~\ref{t31}, one can apply Corollary~\ref{c41}~(2), (3), and (4) to show the inequality $|\lambda_{i}|\leq |\lambda|.$ 
\end{proof}

\begin{corollary}[\text{cf. Corollary~\ref{c32}}]\label{c42}
Let $\{\lambda_{i}\}_{i=1,\ldots,r}$ be an SMB of $\phi[u^{n}].$
\begin{enumerate}[\rm{(}1)]
\item If $n$ is large enough so that $|u^{n}|_{\infty}\geq |\omega_{r}^{0}|/|\omega_{r'+1}^{0}|,$ then for $i=1,\ldots,r,$ we have $|\lambda_{i}| = |\omega_{i}|.$
\item For any positive integer $n,$ we have $|\lambda_{r}|/|\lambda_{r'+1}|\geq |\omega_{r}^{0}|/|\omega_{r'+1}^{0}|.$
\item If $n$ is large enough so that $|u^{n}|_{\infty} > |\omega_{r}^{0}|/|\omega_{r'+1}^{0}|,$ then we have $|\lambda_{i}| < |\omega_{r'+1}^{0}|$ for $i=1,\ldots,r.$
\end{enumerate}
\end{corollary}

\begin{proof}
The equation $|\lambda_{i}|=|\omega_{i}|$ for $i=1,\ldots,r'$ follows from Lemma~\ref{l44}.
Similarly to the proof of Corollary~\ref{c32}, one can apply Corollary~\ref{c41}~(1), Theorem~\ref{t41}, and Proposition~\ref{p221}~(2) to show the rest of the lemma.
\end{proof}

Put $B \coloneqq \{\omega\in \Bbb{C}_{w} \mid |\omega|<|\omega_{r'+1}^{0}|\}.$ 
Since $B\cap \Lambda = \varnothing,$ the exponential function $e_{\phi}$ is injective on $B.$
By (\ref{f41}), we have $|e_{\phi}(\omega)|=|\omega|$ for $\omega\in B.$
This implies $e_{\phi}(B)\subset B.$
Put $C:=e_{\phi}(B).$
There is an inverse $\log_{\phi}:C\to B$ of $e_{\phi}$ defined by a power series with coefficients in $K$ and $e_{\phi}:B \rightleftarrows C:\log_{\phi}$ are inverse to each other.

\begin{lemma}[\text{cf. Lemma~\ref{l32}}]\label{l46}
\begin{enumerate}[\rm{(}1)]
\item We have $C\cap \phi[u^{n}]=B\cap \phi[u^{n}].$
\item We have the following maps which are inverse to each other
\[e_{\phi}:B\cap \mathcal{L}\rightleftarrows B\cap \phi[u^{n}]:\log_{\phi},\]
where \[\mathcal{L}\coloneqq \bigg\{ \sum_{i}a_{i} \cdot_{\psi} \omega_{i}\,\bigg|\,a_{i}\in A\text{ with }\deg(a_{i})<\deg(u^{n}) \bigg\}\]
is a set of representatives of all elements in $u^{-n}\Lambda/\Lambda.$
\item For any $\lambda\in B\cap \phi[u^{n}],$ we have $|\log_{\phi}(\lambda)|=|\lambda|.$
\end{enumerate}
\end{lemma}

Following the strategy of the proof of (1), we can show Lemma~\ref{l46}~(1) alternatively.

\begin{proof}
(1) We know $C\cap\phi[u^{n}] \subset B\cap\phi[u^{n}],$ which implies $\#B\cap\phi[u^{n}] \geq \#C\cap\phi[u^{n}],$ where $\#B\cap\phi[u^{n}]$ denotes the cardinality of the set $B\cap\phi[u^{n}].$
We show \[\#C\cap\phi[u^{n}] \geq \#B\cap\mathcal{L} \geq \#B\cap\phi[u^{n}] \geq \#C\cap\phi[u^{n}].\]
As $e_{\phi}$ is injective on $\mathcal{L},$ we have $\#B\cap\mathcal{L} \leq \#C\cap\phi[u^{n}]$ and it remains to show $\#B\cap\mathcal{L} \geq \#B\cap\phi[u^{n}].$

Put $B^{c}:=\{\omega\in \Bbb{C}_{w} \mid |\omega|\geq |\omega_{r'+1}^{0}|\},$ which is complementary to $B$ in $\Bbb{C}_{w}.$
For any $\omega=\sum_{j}a_{j} \cdot_{\psi} \omega_{j}\in B^{c}\cap \mathcal{L},$ there exists $a_{j}\neq 0$ for some $j>r',$ for otherwise we have $|\omega|<0<|\omega_{r'+1}^{0}|$ by Lemma~\ref{l43}~(1). 
By Corollary~\ref{c41}~(1), we have 
\[|e_{\phi}(\omega)| \geq |\omega|\geq |\omega_{r'+1}^{0}|.\]
Hence $e_{\phi}(B^{c}\cap \mathcal{L})\subset B^{c}\cap \phi[u^{n}].$
As $e_{\phi}$ is injective on $\mathcal{L},$ we have $\# B^{c} \cap \mathcal{L}\leq \#B^{c} \cap \phi[u^{n}].$
This implies $\#B \cap \mathcal{L} \geq \#B \cap \phi[u^{n}],$ as desired.

(2) The map $e_{\phi}:B\cap\mathcal{L} \to B\cap\phi[u^{n}]$ is injective and is also surjective as $\#B\cap\mathcal{L}=\#B\cap\phi[u^{n}].$
Hence (2) follows.

(3) By (2), we have $\log_{\phi}(\lambda)\in B\cap\mathcal{L}$ and $e_{\phi}(\log_{\phi}(\lambda)) = \lambda.$ 
Hence (3) follows from Lemma~\ref{l44} and Corollary~\ref{c41}~(1).
\end{proof}

\begin{lemma}\label{l47}
Let $\{\lambda_{i}\}_{i=1,\ldots,r}$ be an SMB of $\phi[u^{n}].$
We have $w(\lambda_{i})\geq 0$ for $i\leq r'$ and $w(\lambda_{i})<0$ for $i>r'.$
\end{lemma}

\begin{proof}
For a positive integer $j,$ let $\{\lambda_{i,j}\}_{i=1,\ldots,r}$ be an SMB of $\phi[u^{j}]$ as in Corollary~\ref{c23}.
By Proposition~\ref{p221}~(2), we have $w(\lambda_{i})=w(\lambda_{i,n})$ for all $i.$
It suffices to show $w(\lambda_{r',n})\geq 0$ and $w(\lambda_{r'+1,n})<0.$

We first show $w(\lambda_{r',1})\geq 0$ and $w(\lambda_{r'+1,1})<0.$
Put $d \coloneqq \deg(u),$ $u_{0} \coloneqq u,$ $\sum_{i=0}^{rd}u_{i}X^{q^{i}} \coloneqq \phi_{u}(X),$ and $P_{i} \coloneqq (q^{i},w(u_{i}))$ for $i=0,\ldots,rd.$
As $\phi$ has stable reduction, we have $w(u_{i})\geq 0$ for all $i,$ $w(u_{r'd})=0,$ and $w(u_{i})>0$ for all $i>r'd.$
Hence the point $P_{r'd}$ is a vertex of the Newton polygon of $\phi_{u}(X).$
The segments on the left (resp. right) of $P_{r'd}$ have slopes $\leq 0$ (resp. slopes $>0$).
Hence there are exactly $q^{r'd}$ roots with valuations $\geq 0.$ 
Here $0 \in \phi[u]$ is considered to have valuation $>0.$

We show $w(\lambda_{r',1})\geq 0$ and $w(\lambda_{r'+1,1}) < 0$ by induction.
By (\ref{f23}), we have $w(\lambda_{1,1})\geq 0.$
Fix a positive integer $k\leq r'$ and assume $w(\lambda_{i,1})\geq 0$ for $i<k.$ 
Then the elements $\lambda_{i,1}$ for $i<k$ generates an $A/u$-vector subspace of $\phi[u]$ containing $q^{(k-1)d}$ many elements.
Since $\phi$ has stable reduction, for any $a\in A,$ all coefficients of $\phi_{a}(X)$ have valuation $\geq 0.$ 
By the ultrametric inequality, we have $w(a\cdot_{\phi} \lambda_{i,1})$ for any $a\in A \mod u$ and $i < k.$ 
Hence all the elements in the vector subspace have valuations $\geq 0$ .
Since $q^{(k-1)d}<q^{r'd},$ there are elements in $\phi[u]\setminus \bigoplus_{i<k}(A/u)\cdot_{\phi} \lambda_{i,1}$ having valuation $\geq 0.$
By (\ref{f23}), we have $w(\lambda_{k,1})>0.$
For $k=r'+1,$ we have the same inductive hypothesis as above.
However, since $q^{(k-1)d}=q^{r'd},$ each element in $\phi[u]\setminus \bigoplus_{i<k}(A/u)\cdot_{\phi} \lambda_{i,1}$ has valuation $<0$ and hence $w(\lambda_{r'+1,1})<0.$

Next, we show $w(\lambda_{r',n})\geq 0$ (resp. $w(\lambda_{r'+1,n})<0$) by induction.
Assume $w(\lambda_{r',j-1})\geq 0$ (resp. $w(\lambda_{r'+1,j-1})<0$). 
By Corollary~\ref{c23}, the element $\lambda_{r',j}$ (resp. $\lambda_{r'+1,j}$) is a root of $\phi_{u}(X)-\lambda_{r',j-1}$ (resp. $\phi_{u}(X)-\lambda_{r'+1,j-1}$) having the largest valuation.
By the induction hypothesis and the valuations of the coefficients of $\phi_{u}(X),$ the left-most segment in the Newton polygon of $\phi_{u}(X)-\lambda_{r',j-1}$ (resp. $\phi_{u}(X)-\lambda_{r'+1,j-1}$) has slope $\leq 0$ (resp. $>0$).
Hence we have $w(\lambda_{r',j})\geq 0$ and $w(\lambda_{r'+1,j})<0.$
\end{proof}

Let $\{\lambda_{i}\}_{i=1,\ldots,r}$ denote an SMB of $\phi[u^{n}].$
Assume that the positive integer $n$ is large enough so that $|u^{n}|_{\infty}>|\omega_{r}^{0}|/|\omega_{r'+1}^{0}|.$
By Corollary~\ref{c42}~(3) and Lemma~\ref{l46}~(1), for each $i,$ we have $\lambda_{i}\in B\cap \phi[u^{n}] = C \cap \phi[u^{n}]$ and we put $\omega_{i}':=\log_{\phi}(\lambda_{i}).$

\begin{theorem}[\text{cf. Theorem~\ref{t32}}]\label{t42}
\begin{enumerate}[\rm{(}1)]
\item The family of elements $\{\omega_{i}'\}_{i=1,\ldots,r'}$ is an SMB of $\psi[u^{n}].$
\item The family of elements $\{u^{n} \cdot_{\psi} \omega_{i}'\}_{i=r'+1,\ldots,r}$ is an SMB of $\Lambda.$
\end{enumerate}
\end{theorem}

\begin{proof}
(1) To check Definition~\ref{d11}~(1), we show that the elements $\omega_{i}'$ for $i\leq r'$ belong to $\psi[u^{n}]$ and are $A/u^{n}$-linearly independent.
By Lemma~\ref{l46}~(3) and Lemma~\ref{l47}, we have $w(\omega_{i}')=w(\lambda_{i})\geq 0$ for $i\leq r'.$
By Lemma~\ref{l41}~(3), this implies that $\omega_{i}'\in \psi[u^{n}]$ for $i\leq r'.$
Note that $\mathcal{E}_{\phi}: u^{-n}\Lambda/\Lambda \to \phi[u^{n}]$ is an $A/u^{n}$-module isomorphism induced by $e_{\phi}$ and $e_{\phi}(\omega_{i}')=\lambda_{i}.$
If $\sum_{i \leq r'}a_{i} \cdot_{\psi} \omega_{i}'=0$ with $a_{i}\in A \mod u^{n},$ then we have $\sum_{i \leq r'} a_{i} \cdot_{\phi} \lambda_{i} = 0.$
This implies $a_{i}\equiv 0\mod u^{n}$ and hence the desired linear independence.

As $\{\lambda_{i}\}_{i=1,\ldots,r}$ is an SMB of $\phi[u^{n}],$ we can straightforwardly check Definition~\ref{d11}~(2) using Lemma~\ref{l44}. 

(2) 
Similarly to (1), we can apply Lemma~\ref{l47}, \ref{l46}~(3), \ref{l41}~(3) to show $\omega_{i}'\notin \psi[u^{n}]$ such that $u^{n} \cdot_{\psi} \omega_{i}'$ for $i > r'$ belong to $\Lambda.$ 
We check the two dots in Proposition~\ref{p211}~(1).
Let us show that $\omega_{r'+1}',\ldots,\omega_{r}'$ are $A$-linearly independent first.
If there exist $a_{i}\in A$ such that $\sum_{i> r'}a_{i} \cdot_{\psi} \omega_{i}'=0,$ we can show $a_{i}\equiv 0\mod u^{n}$ for all $i$ similarly to (1).
Assume $a_{i}\neq 0$ for some $i.$
Let $m$ be the integer such that $u^{m}\mid a_{i}$ for all $i>r'$ and $u^{m+1}\nmid a_{i}$ for some $i.$
Then there exist $b_{i}\in A$ such that $a_{i}=b_{i}u^{m}$ for all $i>r'$ and $b_{i}\not\equiv 0\mod u$ for some $i.$
Hence $\sum_{i>r'}b_{i} \cdot_{\psi} \omega_{i}'$ is a root of $\psi_{u^{m}}(X)$ and we denote this root by $\omega.$
On the other hand, 
\[u^{n} \cdot_{\psi} \omega = \sum_{i>r'}b_{i}\cdot_{\psi} \left(u^{n} \cdot_{\psi} \omega_{i}'\right)\in \Lambda.\]
Since $\Lambda\cap \psi[u^{m}]=0,$ we have $u^{n} \cdot_{\psi} \omega=0$ and hence $\omega \in \psi[u^{n}].$
By (1), there exist $b_{i}\in A\mod u^{n}$ for $i\leq r'$ such that $\omega = \sum_{i\leq r'} b_{i} \cdot_{\psi} \omega_{i}'.$
This equality implies 
\[0 = e_{\phi}\bigg(\sum_{i>r'}b_{i} \cdot_{\psi} \omega_{i}' - \sum_{i\leq r'}b_{i} \cdot_{\psi} \omega_{i}'\bigg) = \sum_{i>r'} b_{i} \cdot_{\phi} \lambda_{i} - \sum_{i\leq r'}b_{i} \cdot_{\phi} \lambda_{i}.\]
As some $b_{i}\not\equiv 0 \mod u^{n},$ this is absurd.

Finally, we check the second dots in Proposition~\ref{p211}~(1).
Let $l_{r'+1}\leq \cdots\leq l_{r}$ denote the invariant of $\Lambda$ as in Proposition~\ref{p211}~(2).
Fix $i$ to be a positive integer satisfying $r' < i \leq r.$
It suffices to show $l_{i}=|u^{n} \cdot_{\psi} \omega_{i}'|.$
We have $l_{i}\leq |u^{n} \cdot_{\psi} \omega_{i}'|.$
Let us assume $l_{i}<|u^{n} \cdot_{\psi} \omega_{i}'|.$
Since $\lambda_{i} \in B \cap \phi[u^{n}],$ we have $|\omega_{i}'|=|\lambda_{i}|$ by Lemma~\ref{l46}~(3).
Hence $l_{i}/|u^{n}|_{\infty}<|\omega_{i}'|=|\lambda_{i}|<|\omega_{r'+1}^{0}|.$
By Proposition~\ref{p211}, there is an SMB $\{\eta_{j}^{0}\}_{j=r'+1,\ldots,r}$ of $\Lambda$ such that $|\eta_{i}^{0}|=l_{i}.$
Let $\eta_{j}$ be a root of $\psi_{u^{n}}(X)-\eta_{j}^{0}$ for all $j$ (cf. the definition of $\omega_{j}$ before Lemma~\ref{l43}). 
As $|\eta_{i}|=l_{i}/|u^{n}|_{\infty}<|\omega_{r'+1}^{0}|,$ we have $|e_{\phi}(\eta_{i})|=|\eta_{i}|$ by (\ref{f41}).
This implies 
\[|e_{\phi}(\eta_{i})|=|\eta_{i}|=l_{i}/|u^{n}|_{\infty}<|\omega_{i}'|=|\lambda_{i}|.\]
By Theorem~\ref{t41}, the elements $e_{\phi}(\omega_{j}')$ for $j=1,\ldots,r'$ and $e_{\phi}(\eta_{j})$ for $j=r'+1,\ldots,r$ form an SMB of $\phi[u^{n}].$
By Proposition~\ref{p221}~(2), this contradicts $|e_{\phi}(\eta_{i})|<|\lambda_{i}|.$ 
\end{proof}

\begin{proposition}[\text{cf. Proposition~\ref{p31}}]\label{p41}
If $n$ is large enough so that $|u^{n}|_{\infty}\geq |\omega_{r}^{0}|/|\omega_{r'+1}^{0}|,$ then we have 
\[K(u^{-n}\Lambda)=K(\phi[u^{n}]),\]
where $K(u^{-n}\Lambda)$ \text{\rm{(}}resp. $K(\phi[u^{n}])$\text{\rm{)}} is the extension of $K$ generated by all elements in $u^{-n}\Lambda$ \text{\rm{(}}resp. in $\phi[u^{n}]$\text{\rm{)}}.
\end{proposition}

\begin{proof}
Note that $e_{\phi}$ is given by a power series with coefficients in $K$ and it induces an isomorphism $\mathcal{E}_{\phi}:u^{-n}\Lambda/\Lambda\to \phi[u^{n}].$
Similarly to the proof of Proposition~\ref{p31}, one can show $K(\phi[u^{n}])\subset K(u^{-n}\Lambda).$

Note that $\log_{\phi}$ is given by a power series with coefficients in $K.$
For any $y\in C \cap K^{\mathrm{sep}},$ we have $\log_{\phi}(y)\in K(y).$
Let $\{\lambda_{i}\}_{i=1,\ldots,r}$ be an SMB of $\phi[u^{n}].$
As $|u^{n}|_{\infty} > |\omega_{r}^{0}|/|\omega_{r'+1}^{0}|,$ by Theorem~\ref{t42}, the families $\{\omega_{i}'\}_{i=1,\ldots,r'}$ and $\{u^{n} \cdot_{\psi} \omega_{i}'\}_{i=r'+1,\ldots,r}$ are respectively the SMB of $\psi[u^{n}]$ and $\Lambda,$ where $\omega_{i}' = \log_{\phi}(\lambda_{i}).$ 
Since $K(\omega_{i}')\subset K(\lambda_{i})$ for each $i,$ it suffices to show that $\omega_{i}'$ for all $i$ form a generating set of $u^{-n}\Lambda.$ 
For any $\omega \in u^{-n}\Lambda,$ it is a root of $\psi_{u^{n}}(X)-u^{n} \cdot_{\psi} \omega.$
Note $u^{n}\cdot_{\psi}\omega \in \Lambda.$
Since $\{u^{n} \cdot_{\psi} \omega_{i}'\}_{i=r'+1,\ldots,r}$ is an SMB of $\Lambda,$ we have $u^{n}\cdot_{\psi}\omega = \sum_{i>r'}a_{i}\cdot_{\psi}(u^{n}\cdot_{\psi}\omega_{i}')$ for some $a_{i} \in A.$
Hence $\sum_{i>r'}a_{i}\cdot_{\psi}\omega_{i}'$ is also a root of $\psi_{u^{n}}(X)-u^{n} \cdot_{\psi} \omega.$
Since $\{\omega_{i}'\}_{i=1,\cdots,r'}$ is an SMB of $\psi[u^{n}],$ we have $\sum_{i>r'}a_{i}\cdot_{\psi}\omega_{i}' - \omega = \sum_{i\leq r'}a_{i}\cdot_{\psi}\omega_{i}'$ for some $a_{i} \in A \mod u^{n}$ and the claim follows.
\end{proof}

Combining Corollary~\ref{c42}~(2), Theorem~\ref{t42}, and Propostion~\ref{p41}, we have

\begin{corollary}[\text{cf. Corollary~\ref{r32}}]\label{r41}
Let $l$ be a positive integer and $\{\eta_{i}\}_{i=1,\ldots,r}$ an SMB of $\phi[u^{l}].$
Let $\{\lambda_{i}\}_{i=1,\ldots,r}$ be an SMB of $\phi[u^{n}].$
If $n$ is large enough so that $|u^{n}|_{\infty} > |\eta_{r}|/|\eta_{r'+1}|,$ then we have
\begin{enumerate}[\rm{(}1)]
\item the family $\{\log_{\phi}(\lambda_{i})\}_{i=1,\ldots,r'}$ is an SMB of $\psi[u^{n}];$
\item the family $\{u^{n} \cdot_{\psi} \log_{\phi}(\lambda_{i})\}_{i=r'+1,\ldots,r}$ is an SMB of $\Lambda;$
\item $K(u^{-n}\Lambda)=K(\phi[u^{n}]).$
\end{enumerate}
\end{corollary}

\begin{proposition}[\text{cf. Proposition~\ref{p32}}]\label{p42}
Assume $w(u)=0,$ i.e., $u$ is not divisible by the prime $w.$
Let $\{\lambda_{i}\}_{i=1,\ldots,r}$ be an SMB of $\phi[u^{n}].$
Then we have 
\[\bigg|\sum_{i}a_{i} \cdot_{\phi} \lambda_{i}\bigg|=\max_{i}\{|a_{i} \cdot_{\phi} \lambda_{i}|\}\]
for any $a_{i}\in A \mod u^{n}.$
\end{proposition}

\begin{proof}
Assume first that $n$ is large enough so that $|u^{n}|_{\infty}>|\omega_{r}^{0}|/|\omega_{r'+1}^{0}|.$
By Theorem~\ref{t42}, the families $\{\omega_{i}'\}_{i=1,\ldots,r'}$ and $\{u^{n} \cdot_{\psi} \omega_{i}'\}_{i=r'+1,\ldots,r}$ are respectively an SMB of $\psi[u^{n}]$ and $\Lambda,$ where $\omega_{i}'=\log_{\phi}(\lambda_{i}).$ 
Without loss of generality, we assume $\deg(a_{i})<\deg(u^{n}).$
Assume that $a_{i}$ is nonzero for some $i>r'.$ 
By Corollary~\ref{c41}~(3), we have 
\[\bigg|e_{\phi}\bigg(\sum_{i}a_{i} \cdot_{\psi} \omega_{i}'\bigg)\bigg| = \max_{i} \{|a_{i}\cdot_{\phi} e_{\phi}(\omega_{i}')|\}.\]
As $e_{\phi}(\sum_{i}a_{i} \cdot_{\psi} \omega_{i}')=\sum_{i}a_{i} \cdot_{\phi} \lambda_{i},$ the claim follows. 
If $a_{i}=0$ for all $i>r',$ then $\sum_{i\leq r'}a_{i} \cdot_{\psi} \omega_{i}'$ belongs to $\psi[u^{n}].$ 
By Lemma~\ref{l41}~(1), we have $|\sum_{i\leq r'}a_{i} \cdot_{\psi} \omega_{i}'|=0$ and $|a_{i} \cdot_{\psi} \omega_{i}'| = 0$ for all $i \leq r'.$
The desired equality follows from Lemma~\ref{l44}.
Similarly to the proof of Proposition~\ref{p32}, the case where $n$ is arbitrary follows from the case where $n$ is large enough.
\end{proof}

\begin{corollary}[cf. Corollary~\ref{l53}]\label{c43}
Let $u$ be a finite prime of $A$ not divisible by the prime $w.$
Let $\{\lambda_{i}\}_{i=1,\ldots,r}$ be an SMB of $\phi[u^{n}]$ so that $|\lambda_{1}| = \cdots = |\lambda_{s}| < |\lambda_{s+1}|$ for some positive integer $s.$  
\begin{enumerate}[\rm{(}1)]
\item 
The extension of $K$ generated by $\lambda_{1},\ldots,\lambda_{s}$ is unramified.
\item
For an element $\sigma$ in the ramification subgroup $\mathrm{Gal}(K(\phi[u^{n}])/K)_{0},$ we have $\sigma(\lambda_{j}) = \lambda_{j}$ for $j=1,\ldots,s.$ 
\end{enumerate}
\end{corollary}

\begin{proof}
(2) straightforwardly follows from (1).
We show (1).
Notice $w(u) = 0.$
Following the proof of Lemma~\ref{l47}, we know that $s=r'$ and $0 = w(\lambda_{1}) = \cdots = w(\lambda_{r'}) > w(\lambda_{r'+1}).$
Let $n'$ be an integer satisfying $n'\geq n$ and $|u^{n'}|_{\infty} > |\lambda_{r}|/|\lambda_{r'+1}|.$
By Proposition~\ref{p222}~(1), we know that there exists an SMB $\{\lambda_{i}'\}_{i=1,\ldots,r}$ of $\phi[u^{n'}]$ such that $u^{n'-n}\cdot_{\phi}\lambda_{i}' = \lambda_{i}$ for all $i.$
By Theorem~\ref{t42}, with $\omega_{i}'=\log_{\phi}(\lambda_{i}')$ for all $i,$ we have that the families $\{\omega_{i}'\}_{i=1,\ldots,r'}$ and $\{u^{n'} \cdot_{\psi} \omega_{i}'\}_{i=r'+1,\ldots,r}$ are respectively an SMB of $\psi[u^{n'}]$ and an SMB of $\Lambda.$
As $e_{\phi}(u^{n'-n}\cdot_{\psi}\omega_{i}') = \lambda_{i},$ we have $K(\lambda_{i}) \subset K(u^{n'-n}\cdot_{\psi}\omega_{i}')$ for all $i.$
Note that $\{u^{n'-n}\cdot_{\psi}\omega_{i}'\}_{i=1,\ldots,r'}$ is an SMB of $\psi[u^{n}]$ by Proposition~\ref{p222}~(2). 
By \cite[Theorem~6.3.1]{Pap} (initially proved by Takahashi), the extension of $K$ generated by the elements in $\psi[u^{n}]$ is unramified.
The result follows.
\end{proof}

We assumed above that each Drinfeld module has stable reduction over $K.$
For a Drinfeld $A$-module $\phi$ over $K$ which does not have stable reduction, it turns out that $\phi$ is isomorphic to a Drinfeld module having stable reduction over an at worst tamely ramified subextension of $K(\phi[u])/K$ with $u$ not divisible by $w.$ 

\begin{proposition}\label{p43}
Let $\phi$ be a rank $r$ Drinfeld $A$-module over $K$ which does not necessarily have stable reduction.
Let $u$ be a finite prime of $A$ with $w \nmid u$.
Let $r'$ be the positive integer $\leq r$ so that $\phi$ is isomorphic to a Drinfeld module having stable reduction over some extension of $K$ and the reduction has rank $r' \leq r.$
Let $\{\lambda_{i}\}_{i=1,\ldots,r}$ be an SMB of $\phi[u].$
Then $b\phi b^{-1}$ has stable reduction over $K(\lambda_{1})$ for $b = \lambda_{1}^{-1}$ and the extension $K(\lambda_{1})/K$ is at worst tamely ramified.
\end{proposition}

\begin{proof}
Assume $r' < r.$
For $\phi_{t}(X) = tX + \sum_{i=1}^{r}a_{i}X^{q^{i}} \in K[X].$
Let $M$ be a tamely ramified extension of $K$ of degree $q^{r'}-1.$
Let $b$ be an element in $M$ with valuation $w(b) = \frac{w(a_{r'})}{q^{r'}-1}.$ 
Then $b \phi b^{-1}$ has stable reduction over $M.$
The family $\{b\lambda_{i}\}_{i=1,\ldots,r}$ is an SMB of $b\phi b^{-1}[u]$ (Remark~\ref{r22}).
By Corollary~\ref{c43}, the extension $M(\lambda_{1})/M = M(b\lambda_{1})/M$ is unramified.
Hence $M(\lambda_{1})/K$ is tamely ramified and its subextension $K(\lambda_{1})/K$ is at worst tamely ramified.

Following the proof of Lemma~\ref{l47}, we know that $w(b\lambda_{i}) = 0$ for $i=1,\ldots,r'.$
Hence $w(\lambda_{1}) = -\frac{w(a_{r'})}{q^{r'}-1}$ and we may take $b$ to be $\lambda_{1}^{-1}.$

As for the case $r' = r,$ for a tamely ramified extension $M/K$ of degree $q^{r}-1$ and $b\in M,$ 
we have that $b \phi b^{-1}$ has good reduction on $M.$
By \cite[Theorem~6.3.1]{Pap}, the extension $M(b\lambda_{1})/M$ is unramified.
The result for the case $r' = r$ follows similarly. 
\end{proof}

\section{Applications to rank 2 Drinfeld modules, infinite prime case}\label{s5}
Throughout this section, let $w$ be an infinite prime, $u$ a finite prime of $A$ having degree $d,$ and $n$ a positive integer.
Let $\phi$ be a rank $2$ Drinfeld $A$-module over $K$ determined by $\phi_{t}(X)=tX+a_{1}X^{q}+a_{2}X^{q^{2}}\in K[X].$
Let $\bm{j}$ denote the $j$-invariant $a_{1}^{q+1}/a_{2}$ of $\phi.$
Put $w_{0}\coloneqq w(t),$ $w_{1}\coloneqq w(a_{1}),$ and $w_{2}\coloneqq w(a_{2}).$ 
For each positive integer $j,$ let $\{\xi_{i,j}\}_{i=1,2}$ be an SMB of $\phi[t^{j}]$ obtained as in Corollary~\ref{c23}.

\subsection{The valuations of SMBs}\label{s51}
Our goal is to determine the valuations of elements of SMBs of the lattice $\Lambda$ and the module $\phi[u^{n}]$ in terms of $w_{0},$ $w_{1},$ and $w_{2}.$
If $w(\bm{j})<w_{0}q,$ let $m$ be the integer satisfying $w(\bm{j})\in (w_{0}q^{m+1},w_{0}q^{m}].$
By \cite[Lemma~2.1]{AH22}, we have \begin{equation} \begin{split}w(\xi_{1,n}) & = -\left(w_{0}(n-1)+\frac{w_{1}-w_{0}}{q-1}\right)\text{ for }n \geq 1\text{ and} \\
w(\xi_{2,n}) & = \begin{cases} -\frac{w_{2}+w_{1}(q^{n}-q-1)}{(q-1)q^{n}} & 0<n\leq m; \\
-\left(w_{0}(n-m)+\frac{w_{2}+w_{1}(q^{m}-q-1)}{(q-1)q^{m}}\right) & n\geq m.
\end{cases}
\end{split} \label{f51}
\end{equation}
Now the condition $|t^{n}| \geq |\xi_{2,n}|/|\xi_{1,n}|$ in Remark~\ref{r31} reads $-w_{0}n \geq -w(\xi_{2,n}) + w(\xi_{1,n}).$ 
For $n\geq m,$ this inequality is equivalent to \[-w_{0}n\geq -w_{0}(m-1) + \frac{w_{0}}{q-1}-\frac{w(\bm{j})}{(q-1)q^{m}}.\]
For any $n \geq m,$ the inequality $|t^{n}| \geq |\xi_{2,n}|/|\xi_{1,n}|$ holds.
If $w(\bm{j}) \geq w_{0}q,$ by \cite[Proposition~2.4]{AH22}, we have \begin{align} w(\xi_{1,n}) = w(\xi_{2,n}) = -\left(w_{0}(n-1)+\frac{w_{2}-w_{0}}{q^{2}-1}\right). \label{f52}\end{align}
Hence the condition $|t^{n}| \geq |\xi_{2,n}|/|\xi_{1,n}|$ is fulfilled for any positive integer $n.$
\begin{proposition}\label{p51}
Let $\{\omega_{i}\}_{i=1,2}$  be an SMB of $\Lambda$ and $\{\lambda_{i}\}_{i=1,2}$ an SMB of $\phi[u^{n}].$
\begin{enumerate}[\rm{(}1)]
\item If $w(\bm{j}) < w_{0}q$ and $m$ is the integer such that $w(\bm{j}) \in (w_{0}q^{m+1},w_{0}q^{m}],$ we have \begin{equation}\begin{split} w(\omega_{1}) & = w_{0} + \frac{w_{0}}{q-1} - \frac{w_{1}}{q-1} \\ 
w(\omega_{2}) & = w_{0}m + \frac{w(\bm{j})}{(q-1)q^{m}} - \frac{w_{1}}{q-1}.
\end{split} \nonumber\end{equation}
For $n\geq m/d,$ we have $|u^{n}| > |\omega_{2}|/|\omega_{1}|,$ $w(\lambda_{1})=w(\xi_{1,nd}),$ and $w(\lambda_{2})=w(\xi_{2,nd}).$
\item If $w(\bm{j}) \geq w_{0}q,$ we have \[w(\omega_{1})=w(\omega_{2})=w_{0}+\frac{w_{0}}{q^{2}-1}-\frac{w_{2}}{q^{2}-1}.\]
For $n\geq 1,$ we have $w(\lambda_{1})=w(\lambda_{2})=w(\xi_{1,nd})=w(\xi_{2,nd}).$
\end{enumerate}
\end{proposition}

We note that the valuations $w(\omega_{1})$ and $w(\omega_{2})$ above have been obtained by Chen-Lee in \cite[Theorem~3.1 and Corollary 3.1]{CL}. 
One may also recover the rank $r=2$ case of Gekeler's formula \cite[Proposition~3.2]{GekP} (See also \cite[Proposition~5.5.8]{Pap}).

\begin{proof}
The claims of $w(\omega_{1})$ and $w(\omega_{2})$ follow from Remark~\ref{r31}, Corollary~\ref{c32}~(1), and the arguments before the proposition.
Then the claims of $w(\lambda_{1})$ and $w(\lambda_{2})$ are proved by Corollary~\ref{c32}~(1).
\end{proof}

\begin{remark}\label{r52}
Let $r$ be a positive integer and $\phi$ a rank $r$ Drinfeld $A$-module over $K$ such that $\phi_{t}(X) = tX+ a_{s}X^{q^{s}} + a_{r}X^{q^{r}} \in K[X].$
Here $s$ is an positive integer $<r.$
Let $\{\omega_{i}\}_{i=1,\ldots,r}$ be SMB of $\Lambda$ (associated to $\phi$) and $\{\lambda_{i}\}_{i=1,\ldots,r}$ an SMB of $\phi[u^{n}]$ for $u$ and $n$ as above.
Put \[\bm{j} \coloneqq a_{s}^{\frac{q^{r}-1}{q-1}} / a_{r}^{\frac{q^{s}-1}{q-1}}.\]
We obtain the following generalization of Proposition~\ref{p51}. 
Its proof is similar to \cite[Lemma~2.1]{AH22} and Proposition~\ref{p51}:
\begin{enumerate}
\item If $w(\bm{j}) < w_{0}q^{s}\frac{q^{r-s}-1}{q-1}$ and $m$ is the integer such that \[w(\bm{j}) \in \bigg(w_{0}q^{(m+1)s}\frac{q^{r-s}-1}{q-1},w_{0}q^{ms}\frac{q^{r-s}-1}{q-1}\bigg],\] we have 
\begin{align}w(\omega_{i}) & = w_{0} + \frac{w_{0}}{q^{s}-1} - \frac{w_{s}}{q^{s}-1}\text{ for }i=1,\ldots,s,\nonumber\\
w(\omega_{i}) & = w_{0}m + \frac{w(\bm{j})(q-1)}{q^{ms}(q^{s}-1)(q^{r-s}-1)} - \frac{w_{s}}{q^{s}-1}\text{ for }i=s+1,\ldots,r.\nonumber
\end{align}
For $n\geq m/d,$ we have $|u^{n}| > |\omega_{r}|/|\omega_{1}|,$ $w(\lambda_{i}) = - w_{0}nd + w(\omega_{i})$ for $i=1,\ldots,s,$ and $w(\lambda_{i}) = - w_{0}nd + w(\omega_{i})$ for $i=s+1,\ldots,r.$
\item 
If $w(\bm{j}) \geq w_{0}q^{s}\frac{q^{r-s}-1}{q-1},$ we have
\[w(\omega_{i}) = w_{0} + \frac{w_{0}}{q^{r}-1} - \frac{w_{r}}{q^{r}-1}\text{ for }i=1,\ldots,r.\]
For $i=1,\ldots,r$ and $n \geq 1,$ we have $w(\lambda_{i}) = -w_{0}nd + w(\omega_{i}).$
\end{enumerate}

\end{remark}

\subsection{The action of the wild ramification subgroup on the division points}\label{s52}
Let $K(\Lambda)$ (resp. $K(\phi[u^{n}])$) denote the extension of $K$ generated by all elements in $\Lambda$ (resp. in $\phi[u^{n}]$).
If $w(\bm{j})<w_{0}q$ and $m$ is the integer such that $w(\bm{j})\in (w_{0}q^{m+1},w_{0}q^{m}],$ then by Proposition~\ref{p51}~(1) and \ref{p31}, we have for any integer $n\geq m/d$ that (cf. \cite[Lemma~3.3]{AH22}) \begin{align}K(\phi[u^{n}]) = K(\Lambda) = K(\phi[t^{m}]).\label{f53}\end{align}
If $w(\bm{j}) \geq w_{0}q,$ then by Proposition~\ref{p51}~(2) and \ref{p31}, we have for any positive integer $n$ that (cf. \cite[Lemma~3.14]{AH22})
\begin{align}K(\phi[u^{n}]) = K(\Lambda) = K(\phi[t]). \label{f54}\end{align}

Put $G(\Lambda) \coloneqq \mathrm{Gal}(K(\Lambda)/K).$ 
Let $G(\Lambda)_{i}$ and $G(\Lambda)^{y}$ denote respectively the $i$-th lower and $y$-th higher ramification subgroups.
We are to study the action of the wild ramification subgroups $G(n)_{1}$ and $G(\Lambda)_{1}$ on the SMBs of $\phi[u^{n}]$ for $n$ to be large enough. 
Let us recall two lemmas.

\begin{lemma}[\text{\cite[Lemma~3.8]{AH22}}]\label{l51}
Assume $w(\bm{j})<w_{0}q$ and $p \nmid w(\bm{j}).$
Let $m$ be the integer satisfying $w(\bm{j}) \in (w_{0}q^{m+1},w_{0}q^{m}).$
Then we have the (Herbrand) $\psi$-function of the extension $K(\Lambda)/K$ to be 
\begin{equation} \psi_{K(\Lambda)/K}(y) = \begin{cases} y, & -1 \leq y \leq 0;\\
Ey, & 0 \leq y \leq r_{m};\\
q^{j}Ey + w(\bm{j})E\frac{q^{j}-1}{q-1} - w_{0}jEq^{m}, & \begin{split}&r_{m-j+1} \leq y \leq r_{m-j} \\ 
&\quad\text{for }j=1,\ldots,m-1;\end{split}\\
q^{m}Ey + w(\bm{j})E\frac{q^{m}-1}{q-1} - w_{0}mEq^{m}, & r_{1} \leq y,
\end{cases}\nonumber\end{equation}
where 
\begin{align}r_{n} \coloneqq \frac{-w(\bm{j})+w_{0}q^{n}}{q-1}\nonumber\end{align} 
for any positive integer $n \leq m$ and $E$ is some positive integer not divisible by $p.$
\end{lemma}

\begin{lemma}[\text{\cite[Lemma~3.14]{AH22}}]\label{l52}
Assume $w(\bm{j}) \geq w_{0}q.$
Then the extension $K(\phi[t])/K$ is at worst tamely ramified.
\end{lemma}

In Lemma-Definition~\ref{da1}, the conductor of $\phi$ at $w$ is defined to be
\[\mathfrak{f}_{w}(\phi) \coloneqq \int_{0}^{\infty}\left(2-\mathrm{rank}_{A_{u}}T_{u}^{G^{y}}\right) dy.\]
In the next result, we calculate $\mathfrak{f}_{w}(\phi)$ explicitly.
This calculation generalizes the infinite prime case of \cite[Lemma-Definition~4.1]{AH22}.
\begin{lemma}\label{d51}
Assume that one of the following two cases happens
\begin{enumerate}[\rm{(C}1)]
\item $w(\bm{j})<w_{0}q$ and $p\nmid w(\bm{j});$
\item $w(\bm{j})\geq w_{0}q.$
\end{enumerate}
Let $G^{y}$ denote the $y$-th upper ramification subgroup of the Galois group $\mathrm{Gal}(K^{\mathrm{sep}}/K).$
Then we have \[\mathfrak{f}_{w}(\phi)=\begin{cases}\frac{-w(\bm{j})+w_{0}q}{q-1} & \text{if \rm{(C1)}\it{ }happens};\\
0 & \text{if \rm{(C2)}\it{ }happens}.\end{cases}\]
\end{lemma}

\begin{proof}
By Corollary~\ref{c23}, there is an SMB $\{\lambda_{i,n}\}_{i=1,2}$ of $\phi[u^{n}]$ for each integer $n\geq 1$ such that $u \cdot_{\phi} \lambda_{i,n+1}=\lambda_{i,n}$ for $i=1,2.$
Recall that $T_{u}$ is defined to be $\varprojlim_{n}\phi[u^{n}]$ using the morphisms $\phi_{u}:\phi[u^{n+1}]\to \phi[u^{n}]$ for all integers $n\geq 1.$
Hence the tuples $(\lambda_{1,n})_{n\geq 1}$ and $(\lambda_{2,n})_{n\geq 1}$ form an $A_{u}$-basis of $T_{u}.$

Assume (C1) happens.
By (\ref{f53}), the action of $G^{y}$ on $\phi[u^{n}]$ for any $n\geq m/d$ and any $y>0$ factors through $G(\Lambda)^{y}.$
Notice $G(\Lambda)_{1} = \bigcup_{y>0}G(\Lambda)^{y}.$ 
By Corollary~\ref{l53}, any element $\sigma \in G(\Lambda)^{y}$ for $y>0$ fixes $\lambda_{1,n}$ and fixes $u^{j} \cdot_{\phi} \lambda_{1,n} = \lambda_{1,n-j}$ for any positive integer $j < n.$
Hence $\sigma$ fixes $(\lambda_{1,n})_{n\geq 1}.$
As $\lambda_{1,n}$ and $\lambda_{2,n}$ generate $K(\phi[u^{n}])/K = K(\Lambda)/K$ for $n \geq m/d,$ we also have that if $\sigma$ is not the unit, then it nontrivially acts on $\lambda_{2,n}$ and hence nontrivially acts on $(\lambda_{2,n})_{n\geq 1}.$
Therefore Lemma~\ref{l51} implies $\mathrm{rank}_{A_{u}}T_{u}^{G^{y}}=1$ if $0 < y \leq r_{1}$ and $=2$ if $r_{1} < y.$ 
We have
\[\mathfrak{f}_{w}(\phi) = \int_{0}^{r_{1}}1dy = \frac{-w(\bm{j})+w_{0}q}{q-1}.\]

For the case (C2), by (\ref{f54}), the action of $G^{y}$ on $\phi[u^{n}]$ for any $n \geq 1$ and any $y>0$ factors through $G(\Lambda)^{y}.$
By Lemma~\ref{l52}, we have $G(\Lambda)^{y} = \{e\}$ if $y>0.$
The result for the case (C2) immediately follows.
\end{proof}

For an SMB $\{\lambda_{i,n}\}_{i=1,2}$ of $\phi[u^{n}]$ and an element $\sigma\in G(\Lambda)_{1}$ which is not the unit, we work out $\sigma(\lambda_{2,n})$ in the remainder of this subsection.

\begin{lemma}\label{l54}
Assume $w(\bm{j}) \in (w_{0}q^{m+1},w_{0}q^{m})$ for a positive integer $m.$ 
Let $n$ be an integer $\geq m/d$ and $\{\lambda_{i}\}_{i=1,2}$ an SMB of $\phi[u^{n}].$ 
Then we have 
\[w(t^{i} \cdot_{\phi} \lambda_{1}) = w(\xi_{1,nd-i}) \text{ and } w(t^{i} \cdot_{\phi} \lambda_{2}) = w(\xi_{2,nd-i})\text{ for }1 \leq i < nd.\]
\end{lemma}

\begin{proof}
We show the result for $\lambda_{2}.$
The proof of the result for $\lambda_{1}$ is similar. 
By Proposition~\ref{p51}~(1), we have $w(\lambda_{2}) = w(\xi_{2,nd}).$
To know $w(t \cdot_{\phi} \lambda_{2}) = w(t\lambda_{2}+a_{1}\lambda_{2}^{q}+a_{2}\lambda_{2}^{q^{2}}),$ we calculate \begin{align}w(t\lambda_{2})-w(a_{1}\lambda_{2}^{q}) & = \frac{-w(\bm{j})+w_{0}q^{m}((q-1)(nd-m)+1)}{q^{m}},\nonumber\\
w(a_{1}\lambda_{2}^{q}) - w(a_{2}\lambda_{2}^{q^{2}}) & = \frac{w(\bm{j})(q^{m-1}-1)+w_{0}(q-1)(nd-m)q^{m}}{q^{m-1}}<0. 
\nonumber\end{align}
We have \[\begin{cases} w(t\lambda_{2}) - w(a_{1}\lambda_{2}^{q}) > 0 & nd = m;\\
w(t\lambda_{2}) - w(a_{1}\lambda_{2}^{q}) < 0 & nd > m.
\end{cases}\]
Hence we have \[w(t \cdot_{\phi} \lambda_{2}) = \begin{cases} w(a_{1}\lambda_{2}^{q}) = w(\xi_{2,m-1}) & nd = m;\\
w(t\lambda_{2}) = w(\xi_{2,nd-1}) & nd>m.
\end{cases}\]
We assume that the result for $i-1$ is valid.
Put $\lambda_{2}'\coloneqq t^{i-1} \cdot_{\phi} \lambda_{2}.$
If $i\leq nd - m,$ to know $w(t \cdot_{\phi} \lambda_{2}'),$ we calculate \begin{align}& w(t\lambda_{2}')-w(a_{1}\lambda_{2}^{\prime q}) = \frac{-w(\bm{j})+w_{0}q^{m}((q-1)(nd-i-m)+q)}{q^{m}}<0,\nonumber\\
& w(t\lambda_{2}') - w(a_{2}\lambda_{2}^{\prime q^{2}}) \nonumber\\
& \quad\quad = \frac{w(\bm{j})(q^{m}-q-1)+w_{0}q^{m}((q^{2}-1)(nd-i-m)+q^{2})}{q^{m}} < 0. \nonumber\end{align}
Hence we have $w(t \cdot_{\phi} \lambda_{2}') = w(t\lambda_{2}') = w(\xi_{2,nd-i}).$
Assume $i>nd-m.$ 
To know $w(t \cdot_{\phi} \lambda_{2}'),$ we calculate \begin{align} w(t\lambda_{2}') - w(a_{1}\lambda_{2}^{\prime q}) & = \frac{-w(\bm{j})+w_{0}q^{nd-i+1}}{q^{nd-i+1}} > 0,\nonumber\\
w(a_{1}\lambda_{2}^{\prime q}) - w(a_{2}\lambda_{2}^{\prime q^{2}}) & = \frac{w(\bm{j})(q^{nd-i}-1)}{q^{nd-i}} < 0. \nonumber\end{align}
Hence $w(t \cdot_{\phi} \lambda_{2}') = w(a_{1}\lambda_{2}^{\prime q}) = w(\xi_{2,nd-i})$ and the result for $\lambda_{2}$ follows.
\end{proof}

\begin{corollary}\label{c51}
Resume the assumptions in the lemma.
\begin{enumerate}[\rm{(}1)]
\item For any $a\in A$ with $\deg(a) < nd,$ we have \begin{align}
w(a\cdot_{\phi}\lambda_{1}) & = w(t^{\deg(a)}\cdot_{\phi}\lambda_{1}) = w(\xi_{1,nd - \deg(a)}) \label{f57}\\
w(a\cdot_{\phi}\lambda_{2}) & = w(t^{\deg(a)}\cdot_{\phi}\lambda_{2}) = w(\xi_{2,nd - \deg(a)}). \label{f58}
\end{align}
\item For $\lambda\in \phi[u^{n}]$ having valuation $\geq w(\xi_{1,nd-m+1}),$ there exists some $b \in A$ with $\deg(b)<m$ such that $b\cdot_{\phi} \lambda_{1} = \lambda.$
\end{enumerate}
\end{corollary}

\begin{proof}
By \cite[Proposition~2.2]{AH22}, we have \begin{align} & w(\xi_{1,j}) > w(\xi_{2,nd})\text{ for }j=nd,nd-1,\cdots,nd-m+1,\label{f55}\\ 
& w(\xi_{i,j+1}) > w(\xi_{i,j})\text{ for }i=1,2\text{ and positive integers }j < nd.\label{f56}\end{align}
For (1), by (\ref{f56}) and the lemma, we have $w(t^{\deg(a)}\cdot_{\phi}\lambda_{1}) < w(t^{i}\cdot_{\phi}\lambda_{1})$ for any positive integer $i<\deg(a).$
Hence the desired equality follows from the ultrametric inequality.
The equation for $\lambda_{2}$ follows in the same way.

For (2), by (\ref{f55}), we have $w(\lambda) \geq w(\xi_{1,nd-m+1}) > w(\xi_{2,nd}) = w(\lambda_{2}).$ 
As $\{\lambda_{i}\}_{i=1,2}$ is an SMB of $\phi[u^{n}],$ there exist $b,b'\in A\mod u^{n}$ such that $\lambda = b \cdot_{\phi} \lambda_{1} + b' \cdot_{\phi} \lambda_{2}.$ 
We may assume that $b$ and $b'$ have degree $< \deg(u^{n}) = nd.$
Assume conversely $b'\neq 0.$
By (\ref{f58}), we have \[w(b' \cdot_{\phi} \lambda_{2}) = w(t^{\deg(b')} \cdot_{\phi} \lambda_{2}) = w(\xi_{2,nd-\deg(b')}) \leq w(\lambda_{2}).\]
By Proposition~\ref{p32}, we have $w(\lambda) = \min \{w(b \cdot_{\phi} \lambda_{1}), w(b' \cdot_{\phi} \lambda_{2})\}.$
Hence $w(\lambda) \leq w(b' \cdot_{\phi} \lambda_{2}) \leq w(\lambda_{2}),$ a contradiction.
By (\ref{f57}), we have \[w(b \cdot_{\phi} \lambda_{1}) = w(t^{\deg(b)} \cdot_{\phi} \lambda_{1}) = w(\xi_{1,nd-\deg(b)}). \]
Then $w(b \cdot_{\phi} \lambda_{1}) \geq w(\xi_{1,nd-m+1})$ and (\ref{f56}) imply $\deg(b)<m.$
\end{proof}

\begin{remark}
Resume the assumption in the lemma.
\begin{enumerate}[\rm{(}1)]
\item The elements $t^{j}\cdot_{\phi} \lambda_{i}$ for $i=1,2$ and $0 \leq j < nd$ form an $\Bbb{F}_{q}$-basis of $\phi[u^{n}]$ as a vector space.
Indeed, by the lemma and \cite[Proposition~2.2]{AH22}, the valuations $w(t^{j}\cdot_{\phi} \lambda_{i})$ for all $i$ and $j$ are different from each other.
Hence all elements $t^{j} \cdot_{\phi} \lambda_{i}$ are $\Bbb{F}_{q}$-linearly independent and form a $2nd$-dimensional vector subspace of $\phi[u^{n}].$ 
Since $\phi[u^{n}]$ has dimension $2nd$ as an $\Bbb{F}_{q}$-vector space, the claim follows.
\item For a positive integer $j \leq n,$ let $\{\lambda_{i}'\}_{i=1,2}$ be an SMB of $\phi[u^{j}].$
By Corollary~\ref{c51}~(1), we have \[w(\lambda_{1}') = w(\xi_{1,jd})\text{ and }w(\lambda_{2}') = w(\xi_{2,jd}).\]
\end{enumerate}
\end{remark}

Under the assumptions in Lemma~\ref{l51}, we put $R_{i} \coloneqq \psi_{K(\Lambda)/K}(r_{i})$ for $i=1,\ldots,m$ and we have \[R_{i} = -w(\bm{j})E\frac{1}{q-1} - w_{0}Eq^{m}\left(m-i-\frac{1}{q-1}\right).\]

\begin{theorem}\label{t51}
Assume $w(\bm{j}) < w_{0}q$ and $p \nmid w(\bm{j}).$
Let $m$ be the integer such that $w(\bm{j}) \in (w_{0}q^{m+1}, w_{0}q^{m}).$
Let $n$ be an integer $\geq m/d$ and $\{\lambda_{i}\}_{i=1,2}$ an SMB of $\phi[u^{n}].$
Put $G(\Lambda) \coloneqq \mathrm{Gal}(K(\Lambda)/K).$
For a positive integer $i,$ let $A^{<i}$ denote the subgroup of $A$ consisting of elements with degree $<i.$
\begin{enumerate}[\rm{(}1)]
\item Any element in $G(\Lambda)_{1}$ fixes $\lambda_{1};$
\item 
The map \[g: G(\Lambda)_{1} \to A^{<m}\cdot_{\phi}\lambda_{1};\,\,\sigma\mapsto \sigma(\lambda_{2}) - \lambda_{2}\]
is an isomorphism.
\item 
Put $r_{i} \coloneqq \frac{-w(\bm{j}) + w_{0}q^{i}}{q-1}$ for $1 \leq i \leq m$ as in Lemma~\text{\rm{\ref{l51}}}.
Let $G(\Lambda)^{r_{i}}$ denote the $r_{i}$-th upper ramification subgroup of $G(\Lambda).$
Then for each $i=1,\ldots,m,$ the restriction \[g:G(\Lambda)^{r_{i}} \to A^{<i}\cdot_{\phi}\lambda_{1}\] is an isomorphism.
\end{enumerate}
\end{theorem}

\begin{proof}
(1) has been shown in Corollary~\ref{l53}.

(2) We show $\sigma(\lambda_{2}) - \lambda_{2} \in A^{<m}\cdot_{\phi}\lambda_{1}$ for an element $\sigma$ in $G(\Lambda)_{1} = G(\Lambda)_{R_{m}}.$ 
Clearly $\sigma(\lambda_{2}) - \lambda_{2} \in \phi[u^{n}].$ 
By Corollary~\ref{c51}~(2), an element of $\phi[u^{n}]$ having valuation $\geq w(\xi_{1,nd-m+1})$ belongs to the $\Bbb{F}_{q}$-vector space $A^{<m}\cdot_{\phi}\lambda_{1}.$
Hence it suffices to show $w(\sigma(\lambda_{2})-\lambda_{2}) \geq w(\xi_{1,nd-m+1}).$
By Proposition~\ref{p51}, we have $w(\lambda_{i}) = w(\xi_{i,nd}).$
Let $w_{\Lambda}$ denote the normalized valuation associated to $K(\Lambda).$ 
We have $w_{\Lambda} = Eq^{m}w.$
Consider \begin{align} w_{\Lambda}( \sigma(\lambda_{2}) - \lambda_{2} ) & = w_{\Lambda}(\sigma(\lambda_{2})\lambda_{2}^{-1}-1)+w_{\Lambda}(\lambda_{2})\nonumber\\
& \geq R_{m} + w_{\Lambda}(\lambda_{2}) \nonumber\\
& = -w(\bm{j})E\frac{1}{q-1} - w_{0}Eq^{m}\left(-\frac{1}{q-1}\right) \nonumber\\
& \quad\,-Eq^{m}\left(w_{0}(nd-m) + \frac{w_{1}}{q-1} - \frac{w(\bm{j})}{q^{m}(q-1)}\right) \nonumber\\
& = -Eq^{m}\left(w_{0}(nd-m) + \frac{w_{1}-w_{0}}{q-1}\right) = w_{\Lambda}(\xi_{1,nd-m+1}). \nonumber
\end{align}
Hence the image $\sigma(\lambda_{2}) - \lambda_{2}$ of $\sigma$ under the map $g$ belongs to $A^{<m}\cdot_{\phi}\lambda_{1}.$ 

Next, we show that $g$ is an isomorphism.
The map is injective since $\lambda_{1}$ and $\lambda_{2}$ generate $K(\Lambda)/K$ and $\sigma(\lambda_{1})=\lambda_{1}$ for any $\sigma \in G(\Lambda)_{1}.$
By \cite[Theorem~3.9]{AH22}, we know $G(\Lambda)_{1} \cong \Bbb{F}_{q}^{m}.$
As $q^{m}$ is also the cardinal of $A^{<m}\cdot_{\phi}\lambda_{1},$ the map is bijective.
It suffices to show that this map is a morphism.
For any $\sigma \in G(\Lambda)_{1},$ we have that $\sigma$ fixes $\lambda_{1}$ and $\sigma(\lambda_{2}) -\lambda_{2} = b \cdot_{\phi} \lambda_{1}$ for some $b \in A.$
Hence for any $\sigma',$ $\sigma \in G(\Lambda)_{1},$ we have \[\sigma'(\sigma(\lambda_{2})-\lambda_{2})=\sigma(\lambda_{2})-\lambda_{2}.\]
This implies \begin{align} \sigma'(\sigma(\lambda_{2}))-\lambda_{2} & = \sigma'(\sigma(\lambda_{2})) - \sigma'(\lambda_{2}) + \sigma'(\lambda_{2}) - \lambda_{2} \nonumber\\
& = \sigma'(\sigma(\lambda_{2})-\lambda_{2}) + \sigma'(\lambda_{2}) - \lambda_{2} \nonumber\\ 
& = \sigma(\lambda_{2}) - \lambda_{2} + \sigma'(\lambda_{2}) - \lambda_{2}, \nonumber\end{align}
which shows that the map is a morphism.

(3) 
Note $G(\Lambda)^{r_{i}} = G(\Lambda)_{R_{i}}.$
We show that $g:G(\Lambda)_{R_{i}} \to A^{<i}\cdot_{\phi}\lambda_{1}$ is an isomorphism for each $i = 1, \ldots, m.$
By Corollary~\ref{c51}~(1), (\ref{f55}), and (\ref{f56}), the vector space $A^{<i}\cdot_{\phi}\lambda_{1}$ consists of elements of $\phi[u^{n}]$ having valuations $\geq w(\xi_{1,nd-i+1}).$
For $i$ to be one of $1,\ldots,m$ and $\sigma$ to be a nontrivial element in $G(\Lambda)_{R_{i}},$ we have
\begin{align}w_{\Lambda}(\sigma(\lambda_{2}) - \lambda_{2}) & = w_{\Lambda}(\sigma(\lambda_{2})\lambda_{2}^{-1} - 1) + w_{\Lambda}(\lambda_{2})\nonumber\\
& \geq R_{i} + w_{\Lambda}(\lambda_{2}) \nonumber\\
& = -w(\bm{j})E\frac{1}{q-1} - w_{0}Eq^{m}\left(m-i-\frac{1}{q-1}\right) \nonumber\\
& \quad \,-Eq^{m}\left(w_{0}(nd-m) + \frac{w_{1}}{q-1} - \frac{w(\bm{j})}{q^{m}(q-1)}\right) \nonumber\\
& = -Eq^{m}\left(w_{0}(nd-i) + \frac{w_{1}-w_{0}}{q-1}\right) = w_{\Lambda}(\xi_{1,nd-i+1}). \nonumber
\end{align}
This implies that $g(G(\Lambda)_{R_{i}})\subset A^{<i}\cdot_{\phi}\lambda_{1}.$
As the cardinal of $G(\Lambda)_{R_{i}}$ and $A^{<i}\cdot_{\phi}\lambda_{1}$ are both $q^{i},$ the restriction
\[g: G(\Lambda)_{R_{i}} \to A^{<i}\cdot_{\phi}\lambda_{1}\]
is an isomorphism for each $i.$
\end{proof}

\begin{example}[With the help of T. Asayama and Y. Taguchi]\label{e51}
Let $C$ denote the Carlitz $\Bbb{F}_{q}[t]$-module over $\Bbb{F}_{q}(t)$ determined by $\phi_{t}(X) = tX + X^{q}.$
Put $T \coloneqq t^{2} + a$ for some $a \in \Bbb{F}_{q}.$
Let $\phi$ denote the Drinfeld $\Bbb{F}_{q}[T]$-module over $\Bbb{F}_{q}(t)$ determined by $\phi_{T}(X) = C_{t^{2}+a}(X).$
Let $K$ denote the completion of $\Bbb{F}_{q}(t)$ at the infinite prime and $w$ the associated normalized valuation so that $w(t) = -1.$
The $j$-invariant $\bm{j}$ of $\phi$ has valuation $w(\bm{j}) = -q(q+1) < w(t)q.$
However, by \cite[Theorem~7.1.13]{Pap} (initially given by Hayes), the extension $K(\phi[T^{n}]) = K(C[(t^{2}+a)^{n}])$ of $K$ for any $n\in\Bbb{Z}$ is tamely ramified.
By this example, to remove the condition $p \nmid w(\bm{j})$ might be hard.
\end{example}

\section{Application to a rank 2 Drinfeld module, finite prime case}\label{s6}
Let $w$ be a finite prime of $K.$
Throughout this section, let $u$ be a finite prime of $A$ having degree $d,$ and $n$ a positive integer.
Let $\phi$ be a rank $2$ Drinfeld $A$-module over $K$ determined by $\phi_{t}(X)=tX+a_{1}X^{q}+a_{2}X^{q^{2}} \in K[X].$
Let $\bm{j}$ denote the $j$-invariant $a_{1}^{q+1}/a_{2}.$

\subsection{The valuations of SMBs}\label{s61}
Throughout this subsection, assume that $\phi$ has bad reduction over $K,$ i.e., $\phi$ has stable reduction over $K$ and the reduction has rank $1.$ 
We have $w(a_{1})=0$ and $w(a_{2})>0$ such that $w(\bm{j})<0.$
Let $\{\xi_{i,n}\}_{i=1,2}$ be an SMB of $\phi[t^{n}]$ obtained as in Corollary~\ref{c23}.
By \cite[Proposition~2.5 and Lemma~A.1~(2)]{AH22}, we have\begin{align} w(\xi_{1,n}) & = \frac{w(t)}{(q-1)q^{n-1}} \nonumber\\
w(\xi_{2,n}) & = \frac{w(\bm{j})}{(q-1)q^{n}}. \nonumber\end{align}

\begin{proposition}[\text{cf. Proposition~\ref{p51}}]\label{p61}
Let $\{\omega_{1}\}$ be an SMB of $\psi[u^{n}],$ 
$\{\omega_{2}^{0}\}$ an SMB of $\Lambda,$ and $\{\lambda_{i}\}_{i=1,2}$ an SMB of $\phi[u^{n}].$
Then for any positive integer $n,$ we have 
\[w(\omega_{1}) = w(\lambda_{1}) = \frac{w(u)}{(q^{d}-1)q^{(n-1)d}},\,\,w(\omega_{2}^{0}) = \frac{w(\bm{j})}{q-1},\text{ and } w(\lambda_{2}) = \frac{w(\bm{j})}{(q-1)q^{nd}}.\]
\end{proposition}

\begin{proof}
Note that the condition ``$|u^{n}|_{\infty} > |\omega_{r}^{0}|/|\omega_{r'+1}^{0}|$'' in Section~\ref{s4} is trivial.
The results for $\omega_{2}^{0}$ and $\lambda_{2}$ follow from the value $w(\xi_{2,n})$ and Corollary~\ref{c42}~(1).

By Lemma~\ref{l44} and Proposition~\ref{p221}~(2), it remains to calculate $w(\omega_{1}).$
The case $w \nmid u$ is straightforward.
Assume $w \mid u.$
We have $\psi_{t}(X) = tX + b_{1}X^{q} \in 
K[X]$ such that the valuation of $b_{1}$ is $0.$
Let $K'$ denote the extension of $K$ generated by some $b \in K^{\mathrm{sep}}$ with $b^{q-1} = b_{1}.$
Then $C = b\psi b^{-1}$ as Drinfeld $A$-modules over $K'$ where $C$ denotes the Carlitz module.
Let $\{\eta_{1,j}\}$ be an SMB of $C[u^{j}]$ for each positive integer $j$ as in Corollary~\ref{c23}.
As $b\omega_{1}$ forms an SMB of $b\psi b^{-1}[u^{n}],$ we have $w(\omega_{1}) = w(\eta_{1,n})$ by Proposition~\ref{p221}~(2).

To calculate $w(\eta_{1,n}),$ we proceed by induction.
We first calculate $w(\eta_{1,1}).$  
Put $u_{0} \coloneqq u,$ $\sum_{i=0}^{d}u_{i}X^{q^{i}} \coloneqq C_{u}(X),$ and $P_{i}:=(q^{i},w(u_{i}))$ for $i=0,\ldots,d.$
By the explicit formula of $u_{i}$ in \cite[Corollary~5.4.4]{Pap} (initially given by Carlitz), we have $w(u_{i}) = w(u)$ for $i=0,\ldots,d-1.$
The Newton polygon of $C_{u}(X)$ is $P_{0}P_{d}$ having exactly one segment.
Hence we have $w(\eta_{1,1}) = \frac{w(u)}{q^{d}-1}.$ 

Assume $w(\eta_{1,i-1}) = \frac{w(u)}{(q^{d}-1)q^{(i-2)d}}.$
Put $Q_{i-1} \coloneqq (0,w(\eta_{1,i-1})).$
The Newton polygon of $C_{u}(X) - \eta_{1,i-1}$ is $Q_{i-1}P_{d}$ having exactly one segment.
Hence we have $w(\eta_{1,i}) = \frac{w(u)}{(q^{d}-1)q^{(i-1)d}},$ as desired.
\end{proof}

\subsection{The action of the wild ramification subgroup on the division points}\label{s62}
Assume $w \nmid u,$ i.e., $w(u)=0$ throughout this subsection. 
Assume that $\phi$ has bad reduction over $K.$
Let $L$ be the extension of $K$ generated by the elements in $\Lambda$.
For a positive integer $n,$ let $L_{n}$ denote the extension of $L$ generated by the elements in $u^{-n}\Lambda.$
As the condition ``$|u^{n}|_{\infty}>|\omega_{r}^{0}|/|\omega_{r'+1}^{0}|$'' in Section~\ref{s4} is fulfilled for any positive integer $n,$ by Proposition~\ref{p41}, we have $K(\phi[u^{n}]) = L_{n}$ for any positive integer $n.$
We put $G(n) \coloneqq \mathrm{Gal}(K(\phi[u^{n}])/K).$ 

In this subsection, we first study the action of the wild ramification subgroup $G(n)_{1}$ on $u^{-n}\Lambda/\Lambda.$
Next, using the isomorphism $\mathcal{E}_{\phi}:u^{-n}\Lambda/\Lambda \to \phi[u^{n}],$ we know the action of $G(n)_{1}$  on $\phi[u^{n}].$ 
Let $\{\omega_{1}\}$ be an SMB of $\psi[u^{n}],$ $\{\omega_{2}^{0}\}$ an SMB of $\Lambda,$ and $\omega_{2}$ a root of $\psi_{u^{n}}(X) - \omega_{2}^{0}.$

\begin{lemma}\label{c3l52}
The extension $L/K$ is at worst tamely ramified.
\end{lemma}

\begin{proof}
We know that $\Lambda$ is an $A$-lattice via $\psi$ and is $\mathrm{Gal}(K^{\mathrm{sep}}/K)$-invariant.
As $L/K$ is a subextension of $L_{1}/K$ and $L_{1} = K_{1}$ is Galois over $K,$ we have that $L/K$ is separable.
Then the desired claim follows from Lemma~\ref{l21}.
\end{proof}

\begin{theorem}[\text{cf. Theorem~\ref{t51}}]\label{t61}
Let $\phi$ be a rank $2$ Drinfeld $A$-module over $K$ having bad reduction such that $w(\bm{j}) < 0.$ 
Put $R \coloneqq -w(\omega_{2}^{0}) = \frac{-w(\bm{j})}{q-1}.$
Assume $p \nmid w(\bm{j}).$ 
\begin{enumerate}[\rm{(}1)]
\item Let $L(\psi[u^{n}])$ be the extension of $L$ generated by the elements in $\psi[u^{n}].$ 
There is an isomorphism \[\mathrm{Gal}(L_{n}/L(\psi[u^{n}])) \to \psi[u^{n}];\,\,\sigma \mapsto \sigma(\omega_{2}) - \omega_{2}.\]
\item Let $E$ be the ramification index of $L/K.$ 
The (Herbrand) $\psi$-function of the extension $L_{n}/K$ is \[\psi_{L_{n}/K}(y)=\begin{cases} y, & -1 \leq y \leq 0; \\
Ey, & 0 \leq y \leq R; \\
q^{nd}Ey - (q^{nd} - 1)ER, & R \leq y.\end{cases}\]
\end{enumerate}
\end{theorem}

\begin{proof}
Let $w_{L}$ denote the normalized valuation associated to $L.$
We have $w_{L} = Ew.$ 
As the extension $L(\psi[u^{n}])/L$ is unramified, we may let $w_{L}$ denote the normalized valuation associated to $L(\psi[u^{n}]).$
The field $L_{n}$ is the splitting field of $\psi_{u^{n}}(X) - \omega_{2}^{0}$ over $L(\psi[u^{n}]).$
As $E$ is not divisible by $p$ (Lemma~\ref{c3l52}), we have $p \nmid ER = w_{L}(\omega_{2}^{0}).$
Note $w_{L}(\omega_{2}^{0}) < 0.$
We can apply Proposition~\ref{c3p21}~(2) to $\psi_{u^{n}}(X) - \omega_{2}^{0} \in L(\psi[u^{n}])[X].$
Note that the difference between two roots of $\psi_{u^{n}}(X)-\omega_{2}^{0}$ belongs to $\psi[u^{n}].$
The extension $L_{n}/L(\psi[u^{n}])$ is totally ramified and is generated by $\omega_{2}.$ 
The map $\mathrm{Gal}(L_{n}/L(\psi[u^{n}])) \to \psi[u^{n}];\,\,\sigma \mapsto \sigma(\omega_{2}) - \omega_{2}$ is an isomorphism.

(2) By Lemma~\ref{c3l52}, we have the $\psi$-function of $L/K$ to be \[\psi_{L/K}(y) = \begin{cases} y, & -1 \leq y \leq 0;\\
Ey, & 0 \leq y.\end{cases}\]
The $\psi$-function of $L(\psi[u^{n}])/L$ is $\psi_{L(\psi[u^{n}])/L}(y)=y.$
Applying Proposition~\ref{c3p21}~(3) to $\psi_{u^{n}}(X) - \omega_{2}^{0} \in L(\psi[u^{n}]),$ we have  \[\psi_{L_{n}/L(\psi[u^{n}])}(y) = \begin{cases} y, & -1 \leq y \leq ER;\\
q^{nd}y-(q^{nd}-1)ER, & ER\leq y,\end{cases}\]
and the desired $\psi$-function follows from Lemma~\ref{c3l21}.
\end{proof}

Let $\phi$ be a rank $2$ Drinfeld $A$-module over $K$ which does not necessarily have stable reduction.
Assume that $w(\bm{j}) < 0$ such that $\phi$ is isomorphic to a Drinfeld module having bad reduction over some extension of $K.$
By Proposition~\ref{p43}, we may take this extension of $K$ to be $K(\lambda_{1,1}),$ where $\{\lambda_{i,1}\}_{i=1,2}$ is an SMB of $\phi[u]$ and $K(\lambda_{1,1})/K$ is at worst tamely ramified.
Let $\psi$ and $\Lambda$ denote respectively the Drinfeld module having good reduction and the lattice associated to the Drinfeld module having stable reduction via the Tate uniformization.
Let $L$ denote the extension of $K(\lambda_{1,1})$ generated by the elements in $\Lambda.$
By Lemma~\ref{l21}, the extension $L/K$ is at worst tamely ramified.
For a positive integer $n,$ we have $K(\phi[u^{n}]) = L_{n}.$ 

\begin{corollary}\label{c61}
Let $\phi$ be a rank $2$ Drinfeld $A$-module over $K$ which does not necessarily have stable reduction.
Assume $w(\bm{j}) < 0$ and $p \nmid w(\bm{j}).$
\begin{enumerate}[\rm{(}1)]
    \item 
Let $E$ be the ramification index of $L/K.$ 
Put $R = \frac{-w(j)}{q-1}.$
The $\psi$-function of the extension $K(\phi[u^{n}])/K$ is 
\[\psi_{K(\phi[u^{n}])/K}(y) = \begin{cases}
y, & -1 \leq y \leq 0;\\
Ey, & 0 \leq y \leq R;\\
q^{nd}Ey - (q^{nd}-1)ER, & R\leq y.
\end{cases}\]
    \item
Let $\{\lambda_{i}\}_{i=1,2}$ be an SMB of $\phi[u^{n}].$
Then each element in $G(n)_{1}$ fixes $\lambda_{1}$ and there is an isomorphism 
\[G(n)_{1} \to A\cdot_{\phi}\lambda_{1};\,\,\sigma \mapsto \sigma(\lambda_{2}) - \lambda_{2}.\]
\end{enumerate}
\end{corollary}

\begin{proof}
Apply Theorem~\ref{t61}~(2) with $K$ in the theorem being $K(\lambda_{1,1})$ and we obtain the $\psi$-function of $K(\phi[u^{n}])/K(\lambda_{1,1}).$ 
As $K(\lambda_{1,1})/K$ is at worst tamely ramified, its $\psi$-function is clear. 
Then (1) follows from Lemma~\ref{c3l21}.

We show (2).
Note that $L(\psi[u^{n}])/K$ is at worst tamely ramified.
By the $\psi$-function of $L_{n}/K,$ we have the following equation of the higher ramification subgroups
\[G(n)_{1} = \mathrm{Gal}(L_{n}/K)_{1} = \mathrm{Gal}(L_{n}/K)_{ER} = \mathrm{Gal}(L_{n}/L(\psi[u^{n}])).\]
By Proposition~\ref{p43}, the Drinfeld module $b \phi b^{-1}$ for $b = \lambda_{1,1}^{-1}$ has stable reduction over $K(\lambda_{1,1}).$
By Theorem~\ref{t42}, the element $\log_{\phi}(b\lambda_{1})$ forms an SMB of $\psi[u^{n}]$ and $u^{n}\cdot_{\psi}\log_{\phi}(b\lambda_{2})$ forms an SMB of $\Lambda.$
Apply Theorem~\ref{t61}~(1) with $\omega_{1} = \log_{\phi}(b\lambda_{1})$ and $\omega_{2} = \log_{\phi}(b\lambda_{2}).$
We have $\sigma(\log_{\phi}(b\lambda_{1})) = \log_{\phi}(b\lambda_{1})$ for any $\sigma \in G(n)_{1}$ and an isomorphism \[G(n)_{1} \to \psi[u^{n}];\,\,\sigma \mapsto \sigma(\log_{\phi}(b\lambda_{2})) - \log_{\phi}(b\lambda_{2}).\]
Note $\psi[u^{n}] = A \cdot_{\psi} \omega_{1}.$
The map $\mathcal{E}_{b \phi b^{-1}}|_{\psi[u^{n}]}: \psi[u^{n}] \to A\cdot_{b \phi b^{-1}}b\lambda_{1}$ induced by the exponential map $e_{\phi}$ is an isomorphism.
Indeed, it is injective as $\psi[u^{n}] \cap \Lambda = \{0\}.$
Since the sets $\psi[u^{n}]$ and $A\cdot_{b \phi b^{-1}}b\lambda_{1}$ both have cardinal $q^{nd},$ we have the surjectivity.
Notice that $\mathcal{E}_{\phi}$ is compatible with the $\mathrm{Gal}(K^{\mathrm{sep}}/K)$-actions and $\mathrm{Gal}(K^{\mathrm{sep}}/K)$ acts on $u^{-n}\Lambda/\Lambda$ and $\phi[u^{n}]$ via $G(n).$
We obtain that $\sigma(\log_{\phi}(b\lambda_{1})) = \log_{\phi}(b\lambda_{1})$ and $\sigma(\log_{\phi}(b\lambda_{2})) - \log_{\phi}(b\lambda_{2})$ map to respectively $\sigma(b\lambda_{1}) = b\lambda_{1}$ and $\sigma(b\lambda_{2}) - b\lambda_{2}.$
The desired isomorphism is the composition
\[G(n)_{1} \to \psi[u^{n}] \overset{\mathcal{E}_{b \phi b^{-1}}}{\longrightarrow} A\cdot_{b \phi b^{-1}}b\lambda_{1} \overset{b^{-1} \cdot -}{\longrightarrow} A\cdot_{\phi}\lambda_{1};\,\,\sigma \mapsto \sigma(\lambda_{2}) - \lambda_{2}.\]
\end{proof}
The next result generalizes the finite prime case of \cite[Lemma-Definition~4.1]{AH22}.

\begin{lemdef}[\text{cf. Lemma~\ref{d51}}]\label{d61}
Let $\phi$ be a rank $2$ Drinfeld $A$-module over $K$ which does not necessarily have stable reduction.
Assume one of the following two cases happens
\begin{enumerate}[\rm{(C}1)]
\item $w(\bm{j})<0$ and $p\nmid w(\bm{j});$ 
\item $w(\bm{j})\geq 0.$
\end{enumerate}
Let $G^{y}$ denote the $y$-th upper ramification subgroup of the Galois group $\mathrm{Gal}(K^{\mathrm{sep}}/K).$ 
For any finite prime $u$ of $A$ not divisible by $w,$ let $T_{u}$ denote the $u$-adic Tate module of $\phi.$
Put \[\mathfrak{f}_{w}(\phi) \coloneqq \int_{0}^{\infty}\left(2-\mathrm{rank}_{A_{u}}T_{u}^{G^{y}}\right)dy.\]
Then we have
\begin{enumerate}[\rm{(}1)]
\item the value $\mathfrak{f}_{w}(\phi)$ is independent of the choice of $u.$
\item $\mathfrak{f}_{w}(\phi)=\begin{cases}\frac{-w(\bm{j})}{q-1} & \text{\rm{(C1)}}\text{ happens};\\
0 & \text{\rm{(C2)}}\text{ happens}.\end{cases}$
\end{enumerate}
Define the \emph{conductor} $\phi$ at $w$ to be the integral $\mathfrak{f}_{w}(\phi).$
\end{lemdef} 

\begin{proof}
We will show (2) for any finite prime $u$ of $A$ and (1) straightforwardly follows.

Assume the case (C1) happens.
By Corollary~\ref{c23}, there is an SMB $\{\lambda_{i,n}\}_{i=1,2}$ of $\phi[u^{n}]$ for each integer $n\geq 1$ such that $u\cdot_{\phi}\lambda_{i,n+1} = \lambda_{i,n}$ for $i=1,2.$
By Corollary~\ref{c61}~(1), we have $G(n)^{y} = G(n)_{1}$ for any $0 < y \leq \frac{-w(\bm{j})}{q-1}$ and $= \{e\}$ for $y > \frac{-w(\bm{j})}{q-1}.$  
By Corollary~\ref{c61}~(2), for any $n\geq 1$ and $0 < y \leq \frac{-w(\bm{j})}{q-1},$ any element in $G(n)^{y}$ fixes $\lambda_{1,i}$ for all $i\leq n$, and any nontrivial element $\sigma \in G(n)^{y}$ nontrivially acts on $\lambda_{2,n}.$

As $u \cdot_{\phi} \lambda_{1,n+1} = \lambda_{1,n}$ and $u \cdot_{\phi} \lambda_{2,n+1} = \lambda_{2,n}$ for any $n \geq 1,$ the tuples $(\lambda_{1,n})_{n\geq 1}$ and $(\lambda_{2,n})_{n\geq 1}$ form an $A_{u}$-basis of $T_{u}.$
Note that $G^{y}$ acts on $T_{u}$ via $G(\infty)^{y}=\varprojlim_{n}G(n)^{y}.$ 
Any nontrivial element of $G(\infty)^{y}$ for $0 < y \leq \frac{-w(\bm{j})}{q-1}$ fixes $(\lambda_{1,n})_{n\geq 1}$ and nontrivially acts on $(\lambda_{2,n})_{n\geq 1}.$
Hence $\mathrm{rank}_{A_{u}}T_{u}^{G^{y}} = 1$ if $0 < y \leq \frac{-w(\bm{j})}{q-1}$ and $=2$ if $\frac{-w(\bm{j})}{q-1} < y.$
We have \[\mathfrak{f}_{w}(\phi) = \int_{0}^{\frac{-w(\bm{j})}{q-1}} 1dy = \frac{-w(\bm{j})}{q-1}.\]

For the case (C2), we know that $\phi$ is isomorphic to a Drinfeld module having good reduction over some extension of $K.$
By Proposition~\ref{p43}, we may take the extension of $K$ to be $K(\lambda_{1,1})$ and the extension $K(\lambda_{1,1})/K$ is at worst tamely ramified, where $\{\lambda_{i,1}\}_{i=1,2}$ is an SMB of $\phi[u^{n}].$
For $b=\lambda_{1,1},$ as the Drinfeld module $b\phi b^{-1}$ has good reduction, the extension $K(b\phi b^{-1}[u^{n}])/K(\lambda_{1,1})$ is unramified.
Hence the extension $K(\phi[u^{n}])/K$ is at worst tamely ramified and the conductor vanishes.
\end{proof}

\subsection{A function field analogue of Szpiro's conjecture for rank 2 Drinfeld modules}\label{s63}
Let $\phi$ be a rank $2$ Drinfeld $A$-module over $F$ ($F$ is the global function field defined in Section~\ref{s11}).
For a prime $w$ of $F,$ consider $\phi$ as a Drinfeld module over $F_{w}$ and let $\mathfrak{f}_{w}(\phi)$ be the conductor calculated in Lemma~\ref{d51} and Lemma-Definition~\ref{d61}.
Similarly to \cite[Section~4.2]{AH22}, we can obtain a relation between the $J$-height of $\phi$ and the conductors of $\phi.$

Let $M_{F}$ denote the set of primes of $F.$
For a prime $w$ of $F,$ let $\deg(w)$ denote the degree of the residue field of $F_{w}$ over $\Bbb{F}_{q}.$
The $J$-height of $\phi$ is defined to be (See \cite[Section~2.2]{BPR} or \cite[Section~4.2]{AH22})
\[h_{J}(\phi) \coloneqq \frac{1}{[F:\Bbb{F}_{q}(t)]}\sum_{w\in M_{F}}\deg(w)\cdot \max\{-w(\bm{j}),0\},\]
where $\bm{j}$ is the $j$-invariant of $\phi.$
Following \cite[Section~4.2]{AH22}, we may define the (\emph{global}) \emph{conductor} of the Drinfeld module $\phi$ to be 
\[\mathfrak{f}(\phi) \coloneqq \sum_{w\in M_{F}} \deg(w) \cdot \mathfrak{f}_{w}(\phi).\]
Similarly to the proof of \cite[Theorem~4.3]{AH22}, we have the following statement by Lemma~\ref{d51} and Lemma-Definition~\ref{d61}. 
It is a function field analogue of Szpiro's conjecture.
\begin{theorem}\label{t62}
Put $w_{0}=w(t)$ if $w$ is an infinite prime of $F.$
Let $\phi$ be a rank $2$ Drinfeld $A$-module over $F$ such that for each prime $w$ of $F,$ its $j$-invariant $\bm{j}$ satisfies 
\begin{equation}\begin{cases}\begin{split}& \text{either }\big(w(\bm{j}) < w_{0}q\text{ and }p \nmid w(\bm{j})\big),\\
& \text{or }w(\bm{j})\geq w_{0}q\end{split} & \text{ if }w\text{ is infinite};\\
\begin{split}& \text{either }\big(w(\bm{j})<0\text{ and }p\nmid w(\bm{j})\big),\\
& \text{or }w(\bm{j}) \geq 0\end{split} & \text{ if }w\text{ is finite}.
\end{cases}\nonumber\end{equation}
Then 
\[h_{J}(\phi) \leq \mathfrak{f}(\phi)\cdot \frac{q - 1}{[F:\Bbb{F}_{q}(t)]}+q.\]
\end{theorem}

\appendix

\section{The conductors of Drinfeld modules at infinite prime}\label{sa}
Let $K$ be the completion of a global function field at an infinite prime $w.$
Let $\phi$ denote a rank $r$ Drinfeld $A$-module over $K$ for an integer $r \geq 2.$
Let $T_{u}$ be the $u$-adic Tate module of $\phi.$
Let $G^{y}$ denote the $y$-th upper ramification subgroup of the absolute Galois group $G$ of $K.$

\begin{lemdef}\label{da1}
The value of the integral 
\[\int_{0}^{\infty}\left(r-\mathrm{rank}_{A_{u}}T_{u}^{G^{y}}\right)dy\]
is convergent and independent of $u.$
Define the conductor of $\phi$ at $w$ to be this integral.
\end{lemdef}

\begin{proof}
The result follows from the following two claims:
\begin{enumerate}[\rm{(}1)]
    \item 
$\mathrm{rank}_{A_{u}}T_{u}^{G^{y}} = \mathrm{rank}_{A}\Lambda^{G(\Lambda)^{y}}$ for any finite prime $u$ of $A,$ where $G(\Lambda)^{y}$ denotes the $y$-th upper ramification subgroup of the Galois group of the extension of $K$ generated by all elements in $\Lambda.$  
    \item
The following integral is convergent \[\int_{0}^{\infty}\left(r-\mathrm{rank}_{A}\Lambda^{G(\Lambda)^{y}}\right)dy.\]
\end{enumerate}

As for (1), note that the isomorphism $\mathcal{E}_{\phi}:u^{-n}\Lambda/\Lambda \to \phi[u^{n}]$ induced by the exponential map $e_{\phi}$ is $G$-equivariant.
We have a $G$-equivariant isomorphism $T_{u} \cong \Lambda\otimes_{A}A_{u}.$
Note that the action of $G^{y}$ on $\Lambda$ factors through $G(\Lambda)^{y}.$ 
It suffices to see $\mathrm{rank}_{A_{u}}(\Lambda\otimes_{A}A_{u})^{G(\Lambda)^{y}} = \mathrm{rank}_{A}\Lambda^{G(\Lambda)^{y}}.$
As $(\Lambda\otimes_{A}A_{u})^{G(\Lambda)^{y}}$ is free over $A_{u}$ and is identified with $\Lambda^{G(\Lambda)}\otimes_{A}A_{u},$ we know that $\Lambda^{G(\Lambda)^{y}}$ is projective over $A$ according to the following two facts:
\begin{enumerate}
    \item 
The flatness is a local property \cite[00HT]{Sta}; 
    \item
Each finitely generated flat module over a Noetherian ring is projective \cite[00NX]{Sta}.
\end{enumerate}
Then the desired equality follows from the definition of rank and the Nakayama lemma.

As for (2), notice that the extension $K(\Lambda)/K$ is finite.
There is an integer $i$ so that the $i$-th lower ramification subgroup of $\mathrm{Gal}(K(\Lambda)/K)$ is trivial. Hence there is a rational number $\overline{y}$ so that $\overline{y}$-th upper ramification subgroup of $\mathrm{Gal}(K(\Lambda)/K)$ is trivial.
This shows that 
\[\int_{0}^{\infty}\left(r-\mathrm{rank}_{A}\Lambda^{G(\Lambda)^{y}}\right)dy \leq \overline{y}(r-1),\]
i.e. the monotone function \[f(x) \coloneqq \int_{0}^{x}\left(r-\mathrm{rank}_{A}\Lambda^{G(\Lambda)^{y}}\right)dy\] is bounded.
Hence the limit $\lim_{x \to +\infty}f(x)$ exists, i.e., the integral is convergent.
\end{proof}

\begin{remark}
\begin{enumerate}[\rm{(}1)]
    \item
Let $n$ be an integer $\geq m/d.$
Note that $K(\Lambda) = K(\phi[u^{n}]).$ 
Hence one may expect that 
\[\mathrm{rank}_{A_{u}}T_{u}^{G^{y}} = \mathrm{rank}_{A/u^{n}}\phi[u^{n}]^{G^{y}} = \mathrm{rank}_{A/u^{n}}(\Lambda/u^{n}\Lambda)^{G^{y}} = \mathrm{rank}_{A}\Lambda^{G(\Lambda)^{y}}.\]
Here the $A/u^{n}$-submodule $\phi[u^{n}]^{G^{y}}$ of $\phi[u^{n}]$ is free by \cite[VII.14 Theorem 1]{Bou}.
    \item 
Let $w$ be a finite prime, $u$ a finite prime of $A$ with $w \nmid u.$ 
Let $G^{y}$ denote the $y$-th upper ramification subgroup of the absolute Galois group $K.$
M. Mornev has proved that \cite[Theorem~1]{Mor} there is some rational number $\overline{y}$ so that $G^{\overline{y}}$ trivially acts on $T_{u}.$
Using this result, similarly to the proof above, one can show that the integral 
\[\mathfrak{f}_{w}(\phi) = \int_{0}^{\infty}\left(r-\mathrm{rank}_{A_{u}}T_{u}^{G^{y}}\right)dy\]
is convergent.
\end{enumerate}
\end{remark}

\section{Basics of Herbrand $\psi$-functions}\label{sb}
Throughout this section, let $K$ be a complete discrete valuation field of characteristic $p$ so that the residue field is a perfect field.
Let us recall the definition of the (Herbrand) $\psi$-function $\psi_{L/K}$ for a finite Galois extension $L/K$ of a complete valuation field of characteristic $p.$
Let $G^{y}$ denote the $y$-th upper ramification subgroup of the Galois group $\mathrm{Gal}(L/K)$ of $L/K.$
By the $\psi$-function of $L/K,$ we mean the real-valued function on the interval $[0,+\infty)$ defined as 
\[\psi_{L/K}(y)=\int_{0}^{y}\frac{\# G^{0}}{\# G^{r}}dr.\]
We extend $\psi_{L/K}$ to $[-1,+\infty)$ by letting $\psi_{L/K}(y)=y$ if $-1\leq y\leq 0.$
Then $\psi_{L/K}$ is a continuous and piecewise linear function on $[-1,+\infty).$
If $\psi_{L/K}$ is linear on some interval $[a,b]\subset [-1,\infty),$ then we have $G^{b}=G^{y}=G_{\psi_{L/K}(y)}$ for $y\in (a,b].$
By the (maximal) lower ramification break of $L/K,$ we mean the real number $\psi_{L/K}(y),$ where $y\geq 0$ is the maximal real number such that $G^{y}\neq 1.$ 
By the wild ramification subgroup of $L/K,$ we mean the first lower ramification subgroup $G_{1},$ which is equal to the union of $G^{y}$ for $y>0.$

\begin{lemma}[\text{\rm{See e.g., \cite[Chapter~III,\,(3.3)]{FV}}}]\label{c3l21}
Let $L/M$ and $M/K$ be finite Galois extensions. 
Then 
\[\psi_{L/K}=\psi_{L/M}\circ \psi_{M/K}.\]
\end{lemma}

Assume that $K$ contains $\Bbb{F}_{q},$ where $q$ is a power of $p.$
Let $v_{K}$ denote the normalized valuation associated to $K$ so that $v_{K}(K^{\times}) = \Bbb{Z}.$
For a positive integer $s,$ put \[f(X) = X^{q^{s}} + \sum_{k=1}^{s-1}a_{k}X^{q^{k}} + aX\in K[X]\]
such that $\frac{v_{K}(a_{k}) - v_{K}(a)}{q^{k}-1} \geq \frac{-v_{K}(a)}{q^{s}-1}$ for $k=1,\ldots,s-1,$ i.e., the Newton polygon of $f(X)/X$ has exactly one segment.
The extension generated by the roots of the polynomial $f(X)-c$ for certain $c \in K$ plays a key role in Section~\ref{s62}.
To obtain its $\psi$-function, we will need the following fact. It is a slight generalization of the function field case of \cite[Chapter~III,\,Proposition~2.5]{FV} (cf. \cite[Proposition~3.2]{AH22}).

\begin{proposition}\label{c3p21}
Let $f(X)-c$ be the polynomial above.
Let $F$ and $L$ denote respectively the splitting field of $f(X)$ and that of $f(X)-c.$ 
Put $v_{c} \coloneqq v_{K}(c)$ and $v_{a} \coloneqq v_{K}(a).$
Assume $p\nmid v_{c}$ and $\frac{-v_{c}}{q^{s}} < v_{a} - v_{c}$ so that the Newton polygon of $f(X) - c$ has exactly one segment and $R:=\frac{v_{a}q^{s}}{q^{s}-1}-v_{c} > 0.$
Then 
\begin{enumerate}[\rm{(}1)]

\item The extension of $F/K$ is at worst tamely ramified.

\item We have a composition of field extensions 
\[\xymatrix@=1.0em{K \ar@{-}[r] & F \ar@{-}[r] & L.}\]
Moreover, the extension $L/F$ is totally ramified of degree $q^{s}$ and generated by one root $x$ of $f(X)-c.$ 
We have an isomorphism 
\[g:\mathrm{Gal}(L/F) \to V;\quad \sigma \mapsto \sigma(x) - x,\] where $V \cong \Bbb{F}_{q}^{s}$ is the $\Bbb{F}_{q}$-vector space consisting of the roots of $f(X).$ 

\item Let $e$ denote the ramification index of $F/K.$
The $\psi$-function of $L/K$ is   
\[\psi_{L/K}(y)=\begin{cases}
y, & -1 \leq y \leq 0;\\
ey, & 0 \leq y \leq R;\\
eq^{s}y-(q^{s}-1)eR, & R\leq y.\end{cases}
\]

\end{enumerate}
\end{proposition}
\begin{proof}
Let $M$ be an extension of $K$ with ramification index $q^{s}-1.$
We can take some $b \in M$ such that $v(b) = \frac{-v_{a}}{q^{s}-1}.$
With $b' = b^{q^{s}},$ modify $f(X)$ to be 
\[f_{1}(X) = X^{q^{s}} + \sum_{k=1}^{s-1}b_{k}X^{q^{k}} + b_{0}X \coloneqq b'f(X/b).\]
We have 
\[v_{K}(b_{0}) = 0\text{ and }v_{K}(b_{k}) = v_{K}(a_{k}) - \frac{v_{a}(q^{s}-q^{k})}{q^{s}-1} \geq 0\text{ for }k=1,\ldots,s-1.\]
Thus $f_{1}(X)$ is a monic polynomial whose reduction is separable.
By Hensel's lemma \cite[Corollary~2.4.5]{Pap}, the extension of $M$ generated by the roots of $f(X)$ is unramified.
Hence the extension of $K$ generated by the roots of $f(X)$ is at worst tamely ramified.
This shows (1).

For (2), note that the difference of any two roots of $f(X)-c$ is a root of $f(X).$ 
The field $F$ is contained in $L$ and $L$ is the extension of $F$ generated by one root of $f(X)-c.$
As the polynomial $f(X)$ is additive, it root form an $\Bbb{F}_{q}$-vector space of dimensional $s,$ denoted $V.$
Let $x$ be a root of $f(X)-c.$
For any $\sigma\in \mathrm{Gal}(L/F),$ the difference $\sigma(x)-x$ is a root of $f(X)$ and hence we obtain a map $g:\mathrm{Gal}(L/F)\to V;$ $\sigma \mapsto \sigma(x)-x.$
The element $\sigma$ is determined by $\sigma(x)$ since $x$ generates $L/F.$ 
Hence the map $g$ is injective.
This implies that $\# \mathrm{Gal}(L/F) \leq q^{s}.$
As the Newton polygon of $f(X)-c$ has exactly one segment, we have $v_{F}(x)=ev_{c}/q^{s},$ where $v_{F}$ denotes the normalized valuation associated to $F$ and $e$ denotes the ramification index of $F/K.$ 
As $p \nmid e,$ $p\nmid v_{c},$ we have $\# \mathrm{Gal}(L/F) = q^{s}.$ 
Therefore, the extension $L/F$ is a totally ramified Galois extension of degree $q^{s}.$  
The map $\mathrm{Gal}(L/F) \to V$ is surjective as the cardinal of $\mathrm{Gal}(L/F)$ is equal to that of $V.$
As each element $\mathrm{Gal}(L/F)$ fixes each element of $V,$ the map $g$ is a morphism.

We show (3).
Let $\pi_{L}$ be a uniformizer of $L.$ 
For a nontrivial element $\sigma$ in $\mathrm{Gal}(L/F),$ as $\sigma(x)/x$ is a unit of $L$ (here $x$ is a root of $f(X)-c$), we have 
\[\sigma(x)/x=u_{F}\epsilon\] for some $\epsilon \in 1+(\pi_{L})$ (the first higher unit group of $L$) and some $u_{F}$ in the unit group of $F.$
Notice \begin{align}\sigma^{2}(x)/x & = \sigma(xu_{F}\epsilon)/x = u_{F}\sigma(\epsilon)\sigma(x)/x = u_{F}^{2}\sigma(\epsilon)\epsilon,\nonumber\\
\sigma^{3}(x)/x & = \sigma(xu_{F}^{2}\sigma(\epsilon)\epsilon)/x = u_{F}^{2}\sigma^{2}(\epsilon)\sigma(\epsilon)\sigma(x)/x = u_{F}^{3}\sigma^{2}(\epsilon)\sigma(\epsilon)\epsilon \text{ and so on}.\nonumber
\end{align}
As the Galois group of $L/F$ is isomorphic to the $\Bbb{F}_{q}$-vector space of dimensional $s,$ the Galois group element $\sigma$ has order $p.$
We have 
\[1 = \sigma^{p}(x)/x = u_{F}^{p}\prod_{k=0}^{p-1}\sigma^{k}(\epsilon).\]
This implies $u_{F}^{p} \equiv 1 \mod (\pi_{L}).$ 
As $p$-th power map is injective on the residue field of $L,$ we have $u_{F} \equiv 1 \mod (\pi_{L}).$ 
Hence $u_{F} \in 1 + (\pi_{F}),$ where $\pi_{F}$ is a uniformizer of $F.$
We know that $\sigma(x)/x \in 1 + (\pi_{L}).$
Hence there exists some $u_{L}$ in the unit group of $L$ and some positive integer $b$ such that \begin{align} \sigma(x)/x \equiv (1 + u_{L}\pi_{L}^{b}) \mod (\pi_{L})^{b+1}. \label{c2f221}\end{align} 

From (2), we know $v_{L}(x) = ev_{c}$ and is prime to $q^{s}$ ($v_{L}$ denotes the normalized valuation associated to $L$).
Hence there exist integers $i,j$ satisfying $v_{L}(x^{i}\pi_{F}^{j})=1.$
Here $i$ is not divisible by $p.$
The element $x^{i}\pi_{F}^{j}$ is a uniformizer of $L.$
By \cite[Chapter~IV, Proposition~5]{Se}, to know the $\psi$-function of $L/F,$ we need to know $v_{L}(\sigma(x^{i}\pi_{F}^{j})/x^{i}\pi_{F}^{j} - 1)$ for all nontrivial Galois group elements $\sigma.$
By (\ref{c2f221}), we know \[\sigma(x^{i}\pi_{F}^{j})/x^{i}\pi_{F}^{j} \equiv (1 + u_{L}\pi_{L}^{b})^{i} \equiv 1 + iu_{L}^{i}\pi_{L}^{b} \mod  (\pi_{L})^{b+1}.\]
On the other hand, as $v_{L}(\sigma(x)-x) = \frac{v_{a}eq^{s}}{q^{s}-1}$ for any nontrivial $\sigma,$ we know $b = v_{L}(\sigma(x)-x)-v_{L}(x) = eR.$
The $\psi$-functions of $F/K$ and $L/F$ are respectively
\begin{align}\psi_{F/K}(y) & = \begin{cases} y, & -1 \leq y \leq 0;\\
ey, & 0\leq y,
\end{cases}\text{ and }\nonumber\\
\psi_{L/F}(y) & = \begin{cases}y, & -1 \leq y \leq  eR;\\
q^{s}y - (q^{s}-1)eR, & eR \leq y.
\end{cases}\nonumber\end{align}
By Lemma~\ref{c3l21}, we obtain the $\psi$-function $\psi_{L/K}$ as the proposition describes.
\end{proof}

\bibliographystyle{amsalpha}
\bibliography{literature/library}

$\textsc{Department of Mathematics, Tokyo Institute of Technology,}\\\textsc{\,\,\,\,\,\,\,\,2-12-1, O-okayama, Meguro-ku, Tokyo 152-8551, Japan.}$
\end{document}